\documentclass[10pt,a4paper]{article}
\usepackage[utf8]{inputenc}
\usepackage{amsmath}
\usepackage{amsfonts}
\usepackage{amssymb}
\usepackage{amssymb}
\usepackage{float}
\usepackage{amsthm}
\usepackage{multirow}
\usepackage{amscd}
\usepackage{amstext}
\usepackage{fancyhdr}
\usepackage{amsthm}
\usepackage[numbers]{natbib}
\usepackage[T1]{fontenc}
\usepackage{enumerate}
\usepackage[american]{babel}
\usepackage{graphicx}
\usepackage{adjustbox}
\usepackage{bbm}
\usepackage[mathscr]{euscript}
\usepackage{mathrsfs}
\usepackage{stmaryrd}
\usepackage{soul}
\usepackage{hyperref}
\usepackage{xspace}
\usepackage[draft,danish]{fixme}
\usepackage{mathtools}
\usepackage{dsfont}
\usepackage{url}
\usepackage[ps,all,dvips]{xy}
\usepackage{placeins}
\SelectTips{cm}{12}
\CompileMatrices

  \makeatletter
 \useshorthands{"}
 \defineshorthand{"-}{\nobreak-\bbl@allowhyphens}
 \makeatother

\newcommand{\EE}{\mathbb{E}} 

\newcommand{\RR}{\ensuremath{\mathbb{R}}}           % The real numbers

\newcommand{\LL}{\mathcal{L}}

\theoremstyle{plain}
\newtheorem{lemma}{Lemma}[section]
\newtheorem{proposition}[lemma]{Proposition}
\newtheorem{theorem}[lemma]{Theorem}
\newtheorem{corollary}[lemma]{Corollary}
\theoremstyle{definition}
\newtheorem{example}[lemma]{Example}

\theoremstyle{definition}

\theoremstyle{remark}

\newtheorem{remark}[lemma]{Remark}

\newcommand{\devnull}[1]{}

\numberwithin{equation}{section}
\author{Richard A. Davis\textsuperscript{$\ast$}, Mikkel Slot Nielsen\textsuperscript{$\dagger$} and Victor Rohde\textsuperscript{$\dagger$}}

\date{Columbia University\textsuperscript{$\ast$} and Aarhus University\textsuperscript{$\dagger$}}
\begin{document}
\title{Stochastic differential equations with a fractionally filtered delay: a semimartingale model for long-range dependent processes}

\maketitle

\begin{abstract}
In this paper we introduce a model, the stochastic fractional delay differential equation (SFDDE), which is based on the linear stochastic delay differential equation and produces stationary processes with hyperbolically decaying autocovariance functions. The model departs from the usual way of incorporating this type of long-range dependence into a short-memory model as it is obtained by applying a fractional filter to the drift term rather than to the noise term. The advantages of this approach are that the corresponding long-range dependent solutions are semimartingales and the local behavior of the sample paths is unaffected by the degree of long memory. We prove existence and uniqueness of solutions to the SFDDEs and study their spectral densities and autocovariance functions. Moreover, we define a subclass of SFDDEs which we study in detail and relate to the well-known fractionally integrated CARMA processes. Finally, we consider the task of simulating from the defining SFDDEs.
\end{abstract}
\section{Introduction}
Models for time series producing slowly decaying autocorrelation functions (ACFs) have been of interest for more than 50 years. Such models were motivated by the empirical findings of Hurst in the 1950s that were related to the levels of the Nile River. Later, in the 1960s, Benoit Mandelbrot referred to a slowly decaying ACF as the Joseph effect or long-range dependence. Since then, a vast amount of literature on theoretical results and applications have been developed. We refer to \cite{beran2016long,long_range,pipiras2017long,samorodnitsky2016stochastic,samorodnitsky2007long} and references therein for further background.

A very popular discrete-time model for long-range dependence is the \textit{autoregressive fractionally integrated moving average} (ARFIMA) process, introduced by \citet{grangerIntroduction} and \citet{hoskingFractional}, which extends the ARMA process to allow for a hyperbolically decaying ACF. Let $B$ be the backward shift operator and for $\gamma>-1$, define $(1-B)^\gamma$ by means of the binomial expansion,
\begin{align*}
(1-B)^\gamma = \sum_{j=0}^\infty\pi_j B^j
\end{align*}
where $\pi_j =\prod_{0<k\leq j}\frac{k-1-\gamma}{k}$. An ARFIMA process $(X_t)_{t\in \mathbb{Z}}$ is characterized as the unique purely non-deterministic process (as defined in \cite[p.~189]{BrockDavis}) satisfying
\begin{align}\label{discARFIMA}
P(B)(1-B)^{\beta}X_t =Q(B) \varepsilon_t,\quad t \in \mathbb{Z},
\end{align}
where $P$ and $Q$ are real polynomials with no zeroes on $\{z \in \mathbb{C}\, :\, \vert z\vert \leq 1\}$, $(\varepsilon_t)_{t\in \mathbb{Z}}$ is an i.i.d.\ sequence with $\mathbb{E}[\varepsilon_0] = 0$, $\mathbb{E}[\varepsilon_0^2]< \infty$, and $\beta \in (0,1/2)$. The ARFIMA equation (\ref{discARFIMA}) is sometimes represented as an ARMA equation with a fractionally integrated noise, that is,
\begin{align}\label{discARFIMAnoise}
P(B) X_t = Q(B) (1-B)^{-\beta}\varepsilon_t,\quad t \in \mathbb{Z}.
\end{align}
In (\ref{discARFIMA}) one applies a fractional filter to $(X_t)_{t\in \mathbb{Z}}$, while in (\ref{discARFIMAnoise}) one applies a fractional filter to $(\varepsilon_t)_{t\in \mathbb{Z}}$. One main feature of the solution to (\ref{discARFIMA}), equivalently (\ref{discARFIMAnoise}), is that the autocovariance function $\gamma_X(t) := \mathbb{E}[X_0X_t]$ satisfies
\begin{align}\label{longMemory}
\gamma_X (t) \sim c t^{2\beta -1},\quad t \to \infty,
\end{align}
for some constant $c>0$. 

A simple example of a continuous-time stationary process which exhibits long-memory in the sense of (\ref{longMemory}) is an Ornstein-Uhlenbeck process $(X_t)_{t\in \mathbb{R}}$ driven by a fractional L\'{e}vy process, that is, $(X_t)_{t\in \mathbb{R}}$ is the unique stationary solution to
\begin{align}\label{fOUcompact}
dX_t = -\kappa X_t\, dt +d I^\beta L_t,\quad t\in \mathbb{R},
\end{align}
where $\kappa>0$ and
\begin{align}\label{fractionalLevyProc}
I^\beta L_t := \frac{1}{\Gamma (1+\beta)}\int_{-\infty}^t\big[(t-u)^\beta - (-u)_+^\beta \big]\, dL_u,\quad t \in \mathbb{R},
\end{align}
with $(L_t)_{t\in \mathbb{R}}$ being a L\'{e}vy process which satisfies $\mathbb{E}[L_1] = 0$ and $\mathbb{E}[L_1^2]<\infty$. In (\ref{fractionalLevyProc}), $\Gamma$ denotes the gamma function and we have used the notation $x_+ = \max \{x,0\}$ for $x\in \mathbb{R}$. The way to obtain long memory in (\ref{fOUcompact}) is by applying a fractional filter to the noise, which is in line with (\ref{discARFIMAnoise}). To demonstrate the idea of this paper, consider the equation obtained from (\ref{fOUcompact}) but by applying a fractional filter to the drift term instead, i.e.,
\begin{align}\label{newFracOU}
X_t -X_s = -\frac{\kappa}{\Gamma (1-\beta)}\int_{-\infty}^t\big[(t-u)^{-\beta}-(s-u)_+^{-\beta} \big]X_u\, du + L_t -L_s
\end{align}
for $s<t$. One can write (\ref{newFracOU}) compactly as
\begin{align}\label{newFOUcompact}
dX_t = - \kappa D^\beta X_t\, dt + dL_t,\quad t \in \mathbb{R},
\end{align}
with $(D^\beta X_t)_{t\in \mathbb{R}}$ being a suitable fractional derivative process of $(X_t)_{t\in \mathbb{R}}$ defined in Proposition~\ref{fractionalX}. The equations (\ref{newFracOU})-(\ref{newFOUcompact}) are akin to (\ref{discARFIMA}). It turns out that a unique purely non-deterministic process (as defined in (\ref{purelyNon})) satisfying (\ref{newFOUcompact}) exists and has the following properties:
\begin{enumerate}[(i)]
\item\label{lm} The memory is long and controlled by $\beta$ in the sense that $\gamma_X (t)\sim c t^{2\beta -1}$ as $t\to \infty$ for some $c>0$.

\item\label{hc} The $L^2 (\mathbb{P})$-H\"{o}lder continuity of the sample paths is not affected by $\beta$ in the sense that $\gamma_X(0) - \gamma_X (t) \sim c t$ as $t \downarrow 0$ for some $c>0$ (the notion of H\"{o}lder continuity in $L^2(\mathbb{P})$ is indeed closely related to the behavior of the ACF at zero; see Remark \ref{holderRem} for a precise relation).

\item\label{sm} $(X_t)_{t\in \mathbb{R}}$ is a semimartingale.
\end{enumerate}
While both processes in (\ref{fOUcompact}) and (\ref{newFOUcompact}) exhibit long memory in the sense of (\ref{lm}), one should keep in mind that models for long-memory processes obtained by applying a fractional filter to the noise will generally not meet (\ref{hc})-(\ref{sm}), since they inherit various properties from the fractional L\'{e}vy process $(I^\beta L_t)_{t\in \mathbb{R}}$ rather than from the underlying L\'{e}vy process $(L_t)_{t\in \mathbb{R}}$. In particular, this observation applies to the fractional Ornstein-Uhlenbeck process (\ref{fOUcompact}) which is known not to possess the semimartingale property for many choices of $(L_t)_{t\in \mathbb{R}}$, and for which it holds that $\gamma_X(0) - \gamma_X(t) \sim c t^{2\beta +1}$ as $t\downarrow 0$ for some $c>0$ (see \cite[Theorem~4.7]{Tina} and \cite[Proposition~2.5]{QOU}). The latter property, the behavior of $\gamma_X$ near $0$, implies an increased $L^2(\mathbb{P})$-H\"{o}lder continuity relative to (\ref{newFOUcompact}). See Example~\ref{OU-example} for details about the models (\ref{fOUcompact}) and (\ref{newFOUcompact}).

The properties (\ref{hc})-(\ref{sm}) may be desirable to retain in many modeling scenarios. For instance, if a stochastic process $(X_t)_{t\in \mathbb{R}}$ is used to model a financial asset, the semimartingale property is necessary to accommodate the No Free Lunch with Vanishing Risk condition according to the (First) Fundamental Theorem of Asset Pricing, see \cite[Theorem~7.2]{fundametal_thm_asset}. Moreover, if $(X_t)_{t\in \mathbb{R}}$ is supposed to serve as a "good"\ integrator, it follows by the Bichteler-Dellacherie Theorem (\cite[Theorem~7.6]{Bichteler}) that $(X_t)_{t\in \mathbb{R}}$ must be a semimartingale. Also, the papers \cite{bennedsen2015rough,bennedsen2016decoupling} find evidence that the sample paths of electricity spot prices and intraday volatility of the E-mini S\&P500 futures contract are rough, and \citet{jusselin2018no} show that the no-arbitrage assumption implies that the volatility of the macroscopic price process is rough. These findings suggest less smooth sample paths than what is induced by models such as the fractional Ornstein-Uhlenbeck process (\ref{fOUcompact}). In particular, the local smoothness of the sample paths should not be connected to the strength of long memory.

Several extensions to the fractional Ornstein-Uhlenbeck process (\ref{fOUcompact}) exist. For example, it is worth mentioning that the class of \textit{fractionally integrated continuous-time autoregressive moving average} (FICARMA) processes were introduced in \citet{BrockwellMarquardt}, where it is assumed that $P$ and $Q$ are real polynomials with $\text{deg}(P)>\text{deg}(Q)$ which have no zeroes on $\{z\in \mathbb{C}\, :\, \text{Re}(z) \geq 0\}$. The FICARMA process associated to $P$ and $Q$ is then defined as the moving average process
\begin{align}\label{FICARMA1}
X_t = \int_{-\infty}^t g (t-u)\, dI^\beta L_u,\quad t \in \mathbb{R},
\end{align}
with $g :\mathbb{R}\to \mathbb{R}$ being the $L^2$ function characterized by 
\begin{align*}
\mathcal{F}[g](y):=\int_\mathbb{R}e^{iyu}g (u)\, du = \frac{Q(-iy)}{P(-iy)},\quad y\in \mathbb{R}.
\end{align*}
%Based on the characterization (\ref{FICARMA1})-(\ref{FICARMA2}), and in line with (\ref{discARFIMA}), we will think of the FICARMA process as the solution to the formal equation
%\begin{align}\label{contFICARMA}
%P(D) D^\beta X_t = Q(D) DL_t,\quad t\in \mathbb{R},
%\end{align}
%where $D$ denotes differentiation with respect to time and $D^\beta$ is a suitable fractional derivative operator. 
In line with (\ref{discARFIMAnoise}) for the ARFIMA process, a common way of viewing a FICARMA process is that it is obtained by applying a CARMA filter to fractional noise, that is, $(X_t)_{t\in \mathbb{R}}$ given by (\ref{FICARMA1}) is the solution to the formal equation
\begin{align*}
P(D)X_t = Q(D) DI^\beta L_t,\quad t \in \mathbb{R}.
\end{align*}
(See, e.g., \cite{Tina}.) Another class, related to the FICARMA process, consists of solutions $(X_t)_{t\in \mathbb{R}}$ to fractional \textit{stochastic delay differential equations} (SDDEs), that is, $(X_t)_{t\in \mathbb{R}}$ is the unique stationary solution to
\begin{align}\label{usualFracSDDE}
dX_t = \int_{[0,\infty)} X_{t-u}\, \eta (du)\, dt + dI^\beta L_t,\quad t\in \mathbb{R}, 
\end{align}
for a suitable finite signed measure $\eta$. See \cite{contARMAframework,Mohammed} for details about fractional SDDEs. Note that the fractional Ornstein-Uhlenbeck process (\ref{fOUcompact}) is a FICARMA process with polynomials $P(z) = z + \kappa$ and $Q(z) = 1$ and a fractional SDDE with $\eta = -\kappa \delta_0$, $\delta_0$ being the Dirac measure at zero. 
%In line with (\ref{fOUcompact}), both classes meet (\ref{lm}), but not (\ref{hc})-(\ref{sm}). 
% By formally applying the fractional derivative $D^{1+\beta}$ to (\ref{usualFracSDDE}), the equation may be interpreted as
%\begin{align}\label{fracSDDE}
%D\big(D^\beta X \big)_t = \big(D^\beta X \big)\ast \eta (t) + DL_t,\quad t \in \mathbb{R},
%\end{align}
%where $\ast$ denotes convolution of a function and a measure (see Section \ref{Pre}). 

The model we present includes (\ref{newFracOU}) and extends this process in the same way as the fractional SDDE (\ref{usualFracSDDE}) extends the fractional Ornstein-Uhlenbeck (\ref{fOUcompact}). Specifically, we will be interested in a stationary process $(X_t)_{t\in \mathbb{R}}$ satisfying
\begin{align}\label{fractionalSDDEintro}
X_t - X_s = \int_{-\infty}^t  \big(D^\beta_- \mathds{1}_{(s,t]} \big) (u)\, 
\int_{[0,\infty)}X_{u-v}\, \eta (dv)\, du + L_t -L_s
\end{align}
almost surely for each $s<t$, where $\eta$ is a given finite signed measure and
\begin{align*}
\big(D^\beta_- \mathds{1}_{(s,t]} \big) (u) = \frac{1}{\Gamma (1-\beta)}\big[(t-u)_+^{-\beta}- (s-u)_+^{-\beta}\big].
\end{align*}
We will refer to (\ref{fractionalSDDEintro}) as a \textit{stochastic fractional delay differential equation} (SFDDE). Equation (\ref{fractionalSDDEintro}) can be compactly written as
\begin{align}\label{semiMGrepresentation}
dX_t = \int_{[0,\infty)} D^\beta X_{t-u}\, \eta (du)\, dt + dL_t,\quad t \in \mathbb{R},
\end{align}
with $(D^\beta X_t)_{t\in \mathbb{R}}$ defined in Proposition~\ref{fractionalX}. Representation (\ref{semiMGrepresentation}) is, for instance, convenient in order to argue that solutions are semimartingales.

In Section~\ref{fracSDDEmodel} we show that, for a wide range of measures $\eta$, there exists a unique purely non-deterministic process $(X_t)_{t \in \mathbb{R}}$ satisfying the SFDDE (\ref{fractionalSDDEintro}). In addition, we study the behavior of the autocovariance function and the spectral density of $(X_t)_{t\in \mathbb{R}}$ and verify that (\ref{lm})-(\ref{hc}) hold. We end Section~\ref{fracSDDEmodel} by providing an explicit (prediction) formula for computing $\mathbb{E}[X_t \mid X_u,\, u \leq s]$. In Section~\ref{expSection} we focus on delay measures $\eta$ of exponential type, that is,
\begin{align}\label{exponentialIntro}
\eta (du) = -\kappa \delta_0 (du) + f(u)\, du,
\end{align}
where $f(t) = \mathds{1}_{[0,\infty)}(t)b^T e^{At}e_1$ with $e_1= (1,0,\dots, 0)^T\in \mathbb{R}^n$, $b\in \mathbb{R}^n$, and $A$ an $n\times n$ matrix with a spectrum contained in $\{z \in \mathbb{C}\, :\, \text{Re}(z)<0\}$. Besides relating this subclass to the FICARMA processes, we study two special cases of (\ref{exponentialIntro}) in detail, namely the Ornstein-Uhlenbeck type presented in (\ref{newFOUcompact}) and
\begin{align}\label{simpleExample}
dX_t = \int_0^\infty D^\beta X_{t-u} f(u)\, du\, dt+ dL_t,\quad t \in \mathbb{R}.
\end{align}
Equation (\ref{simpleExample}) is interesting to study as it collapses to an ordinary SDDE (cf. Propostion~\ref{fracDiffChange}), and hence constitutes an example of a long-range dependent solution to equation (\ref{usualFracSDDE}) with $I^\beta L_t-I^\beta L_s$ replaced by $L_t -L_s$. While (\ref{simpleExample}) falls into the overall setup of \cite{basse2018multivariate}, the results obtained in that paper do, however, not apply. Finally, based on the two examples (\ref{newFracOU}) and (\ref{simpleExample}), we investigate some numerical aspects in Section~\ref{estimation}, including the task of simulating $(X_t)_{t\in \mathbb{R}}$ from the defining equation. The proofs of all the results presented in Section~\ref{fracSDDEmodel} and \ref{expSection} are contained in the corresponding appendix. We start with a preliminary section which recalls a few definitions and results that will be used repeatedly.

\section{Preliminaries}\label{Pre}
For a measure $\mu$ on the Borel $\sigma$-field $\mathscr{B}(\mathbb{R})$ on $\mathbb{R}$, let $L^p(\mu)$ denote the $L^p$ space relative to $\mu$. If $\mu$ is the Lebesgue measure we suppress the dependence on $\mu$ and write $L^p $ instead of $L^p(\mu)$. By a finite signed measure we refer to a set function $\mu:\mathscr{B}(\mathbb{R})\to \mathbb{R}$ of the form $\mu = \mu^+ - \mu^-$, where $\mu^+$ and $\mu^-$ are two finite singular measures. Integration of a function $f$ with respect to $\mu$ is defined (in an obvious way) whenever $f\in L^1(\vert \mu \vert)$ where $\vert \mu \vert := \mu^+ + \mu^-$. The convolution of two measurable functions $f,g:\mathbb{R}\to \mathbb{C}$ is defined as
\begin{align*}
f\ast g (t) = \int_\mathbb{R}f(t-u)g(u)\, du
\end{align*}
whenever $f(t-\cdot)g \in L^1$. Similarly, if $\mu$ is a finite signed measure, we set
\begin{align*}
f \ast \mu (t) = \int_\mathbb{R}f(t-u)\, \mu (du)
\end{align*}
if $f(t-\cdot) \in L^1(\vert \mu\vert)$. For such $\mu$ set
\begin{align*}
D(\mu) = \biggr\{z \in \mathbb{C}\ :\, \int_\mathbb{R}e^{\text{Re}(z)u}\, \vert \mu\vert (du) <\infty\biggr\}.
\end{align*}
Then we define the bilateral Laplace transform $\mathcal{L}[\mu]:D(\mu)\to \mathbb{C}$ of $\mu$ by
\begin{align*}
\mathcal{L}[\mu](z) = \int_\mathbb{R} e^{zu}\mu (du),\quad z \in D(\mu),
\end{align*}
and the Fourier transform by $\mathcal{F}[\mu](y) = \mathcal{L}[f](iy)$ for $y\in \mathbb{R}$. If $f\in L^1$ we will write $\mathcal{L}[f] = \mathcal{L}[f(u)\, du]$ and $\mathcal{F}[f] = \mathcal{F}[f(u)\, du]$. We also note that $\mathcal{F}[f]\in L^2$ when $f \in L^1 \cap L^2$ and that $\mathcal{F}$ can be extended to an  isometric isomorphism from $L^2$ onto $L^2$ by Plancherel's theorem.

Recall that a L\'{e}vy process is the continuous-time analogue to the (discrete time) random walk. More precisely, a one-sided L\'{e}vy process $(L_t)_{t\geq 0}$, $L_0=0$, is a stochastic process having stationary independent increments and c\'{a}dl\'{a}g sample paths. From these properties it follows that the distribution of $L_1$ is infinitely divisible, and the distribution of $(L_t)_{t\geq 0}$ is determined from $L_1$ via the relation $\mathbb{E}[e^{iyL_t}] = \exp\{t\log \mathbb{E}[e^{iyL_1}]\}$ for $y \in \mathbb{R}$ and $t \geq 0$. The definition is extended to a two-sided L\'{e}vy process $(L_t)_{t\in \mathbb{R}}$ by taking a one-sided L\'{e}vy process $(L^1_t)_{t\geq 0}$ together with an independent copy $(L^2_t)_{t\geq 0}$ and setting $L_t = L^1_t$ if $t\geq 0$ and $L_t =- L^2_{(-t)-}$ if $t<0$. If $\mathbb{E}[L_1^2]<\infty$, $\mathbb{E}[L_1] = 0$ and $f \in L^2$, the integral $\int_\mathbb{R} f(u)\, dL_u$ is well-defined as an $L^2$ limit of integrals of step functions, and the following isometry property holds:
\begin{align*}
\mathbb{E}\biggr[\biggr(\int_\mathbb{R}f(u)\, dL_u \biggr)^2 \biggr] = \mathbb{E}\big[L_1^2\big]\int_\mathbb{R} f(u)^2\, du.
\end{align*}
For more on L\'{e}vy processes and integrals with respect to these, see \cite{Rosinski_spec,Sato}. Finally, for two functions $f,g:\mathbb{R}\to \mathbb{R}$ and $a \in [-\infty,\infty]$ we write $f(t) = o(g(t))$, $f(t) = O(g(t))$ and $f(t)\sim g(t)$ as $t\to a$ if
\begin{align*}
\lim_{t\to a}\frac{f(t)}{g(t)} = 0,\quad \limsup_{t\to a} \biggr\vert \frac{f(t)}{g(t)}\biggr\vert < \infty \quad \text{and}\quad \lim_{t \to a}\frac{f(t)}{g(t)} = 1,
\end{align*}
respectively.

\section{The stochastic fractional delay differential equation}\label{fracSDDEmodel}
Let $(L_t)_{t\in \mathbb{R}}$ be a L\'{e}vy process with $\mathbb{E}[L_1^2]<\infty$ and $\mathbb{E}[L_1] = 0$, and let $\beta \in (0,1/2)$. Without loss of generality we will assume that $\mathbb{E}[L_1^2] = 1$. Moreover, denote by $\eta$ a finite (possibly signed) measure on $[0,\infty)$ with
\begin{align}\label{finiteEta}
\int_{[0,\infty)} u\, \vert \eta\vert (du)< \infty.
\end{align}
and set
\begin{align}\label{fracIndicator}
\big(D^\beta_- \mathds{1}_{(s,t]} \big) (u) = \frac{1}{\Gamma (1-\beta)}\big[(t-u)_+^{-\beta}-(s-u)_+^{-\beta} \big],\quad u \in \mathbb{R}.
\end{align}
(In line with \cite{long_range} we write $D^\beta_-\mathds{1}_{(s,t]}$ rather than $D^\beta \mathds{1}_{(s,t]}$ in (\ref{fracIndicator}) to emphasize that it is the right-sided version of the Riemann-Liouville fractional derivative of $\mathds{1}_{(s,t]}$.) Then we will say that a process $(X_t)_{t \in \mathbb{R}}$ with $\mathbb{E}[\vert X_0\vert ]<\infty$ is a solution to the corresponding SFDDE if it is stationary and satisfies
\begin{align}\label{fractionalSDDE}
X_t - X_s = \int_{-\infty}^t  \big(D^\beta_- \mathds{1}_{(s,t]} \big) (u)\, \int_{[0,\infty)}X_{u-v}\, \eta (dv)\, du + L_t -L_s
\end{align}
almost surely for each $s<t$. Note that equation (\ref{fractionalSDDE}) is indeed well-defined, since $\eta$ is finite, $(X_t)_{t\in \mathbb{R}}$ is bounded in $L^1(\mathbb{P})$ and $D^\beta_- \mathds{1}_{(s,t]} \in L^1$. As noted in the introduction, we will often write (\ref{fractionalSDDE}) shortly as
\begin{align}\label{fractionalSDDEcompact}
dX_t = \int_{[0,\infty)} D^\beta X_{t-u}\, \eta (du)\, dt+ dL_t,\quad t \in \mathbb{R},
\end{align}
where $(D^\beta X_t)_{t\in \mathbb{R}}$ is a suitable fractional derivative of $(X_t)_{t\in \mathbb{R}}$ (defined in Proposition~\ref{fractionalX}).

In order to study which choices of $\eta$ that lead to a stationary solution to (\ref{fractionalSDDE}) we introduce the function $h=h_{\beta,\eta}:\{z \in \mathbb{C}\, :\, \text{Re}(z)\leq 0\}\to \mathbb{C}$ given by
\begin{align}\label{hFunction}
h(z) = (-z)^{1-\beta} - \int_{[0,\infty)} e^{zu}\, \eta (du).
\end{align}
Here, and in the following, we define $z^\gamma = r^\gamma e^{i\gamma \theta}$ using the polar representation $z=r e^{i\theta}$ for $r>0$ and $\theta \in (-\pi,\pi]$. This definition corresponds to $z^\gamma = e^{\gamma \log z}$, using the principal branch of the complex logarithm, and hence $z\mapsto z^\gamma$ is analytic on $\mathbb{C}\setminus \{z \in \mathbb{R}\, :\, z \leq 0\}$. In particular, this means that $h$ is analytic on $\{z \in \mathbb{C}\, :\, \text{Re}(z)<0\}$.
\begin{proposition}\label{kernels} Suppose that $h(z)$ defined in (\ref{hFunction}) is non-zero for every $z\in \mathbb{C}$ with $\text{Re}(z)\leq 0$. Then there exists a unique $g:\mathbb{R}\to \mathbb{R}$, which belongs to $L^\gamma$ for $(1-\beta)^{-1}<\gamma \leq 2$ and is vanishing on $(-\infty,0)$, such that
\begin{align}\label{kernelChar}
\mathcal{F}[g] (y) = \frac{(-iy)^{-\beta}}{h(iy)}
\end{align}
for $y \in \mathbb{R}$. Moreover, the following statements hold:
\begin{enumerate}[(i)]
\item\label{dAlphaLimit} For $t>0$ the Marchaud fractional derivative $D^\beta g (t)$ at $t$ of $g$ given by
\begin{align}\label{dAlphaG}
D^\beta g (t) = \frac{\beta}{\Gamma (1-\beta)}\lim_{\delta \downarrow 0} \int_\delta^\infty \frac{g(t) - g(t-u)}{u^{1+\beta}}\, du
\end{align}
exists, $D^\beta g \in L^1\cap L^2$ and $\mathcal{F}[D^\beta g](y) = 1/h(iy)$ for $y\in \mathbb{R}$.
\item\label{riemannLiouvilleg} The function $g$ is the Riemann-Liouville fractional integral of $D^\beta g$, that is,
\begin{align*}
g(t) = \frac{1}{\Gamma (\beta)}\int_0^t D^\beta g (u) (t-u)^{\beta -1}\, du
\end{align*}
for $t>0$.  
\item\label{differentiability} The function $g$ satisfies
\begin{align}\label{functionalRelation}
g (t) = 1 + \int_0^t \big(D^\beta g \big)\ast \eta (u)\, du,\quad t \geq 0,
\end{align}
and, for $v\in \mathbb{R}$ and with $D^\beta_- \mathds{1}_{(s,t]}$ given in (\ref{fracIndicator}),
\begin{align}\label{functionalEquation}
g(t-v) - g(s-v) = \int_{-\infty}^t \big(D^\beta_- \mathds{1}_{(s,t]} \big) (u)\, g \ast \eta (u-v)\, du + \mathds{1}_{(s,t]}(v).
\end{align}
\end{enumerate}
\end{proposition}

Before formulating our main result, Theorem~\ref{existenceTheorem}, recall that a stationary process $(X_t)_{t\in \mathbb{R}}$ with $\mathbb{E}[X_0^2]<\infty$ and $\mathbb{E}[X_0] = 0$ is said to be purely non-deterministic if 
\begin{align}\label{purelyNon}
\bigcap_{t\in \mathbb{R}} \overline{\text{sp}}\, \{X_s\, :\, s \leq t\} = \{0\},
\end{align} 
see \cite[Section~4]{QOU}. Here $\overline{\text{sp}}$ denotes the $L^2(\mathbb{P})$-closure of the linear span.

\begin{theorem}\label{existenceTheorem} Suppose that $h (z)$ defined in (\ref{hFunction}) is non-zero for every $z\in \mathbb{C}$ with $\text{Re}(z)\leq 0$ and let $g$ be the function introduced in Proposition~\ref{kernels}. Then the process
\begin{align}\label{solution}
X_t = \int_{-\infty}^t g(t-u)\, dL_u, \quad t \in \mathbb{R},
\end{align}
is well-defined, centered and square integrable, and it is the unique purely non-deterministic solution to the SFDDE (\ref{fractionalSDDE}).
\end{theorem}

\begin{remark}\label{uniqueness}
%The uniqueness part in Theorem~\ref{existenceTheorem} states that $(X_t)_{t\in \mathbb{R}}$ is the unique solution to (\ref{fractionalSDDE}) which is stationary, centered, square integrable, and satisfies
%\begin{align}\label{purelyNon}
%\bigcap_{t\in \mathbb{R}} \overline{\text{span}(\{X_s\, :\, s \leq t\})} = \{0\},
%\end{align} 
%see \cite[Section~4]{QOU}. (Here the notation $\overline{A}$ is used for the closure of a subset $A\subseteq L^2(\mathbb{P})$.) 
Note that we cannot hope to get a uniqueness result without imposing a condition such as (\ref{purelyNon}). For instance, the fact that
\begin{align*}
\int_{-\infty}^t\big[(t-u)^{-\beta}-(s-u)_+^{-\beta}\big] \, du= 0,
\end{align*}
shows together with (\ref{fractionalSDDE}) that $(X_t + U)_{t\in \mathbb{R}}$ is a solution for any $U\in L^1(\mathbb{P})$ as long as $(X_t)_{t\in \mathbb{R}}$ is a solution. Moreover, uniqueness relative to condition (\ref{purelyNon}) is similar to that of discrete-time ARFIMA processes, see \cite[Theorem~13.2.1]{BrockDavis}.
\end{remark}

\begin{remark}\label{alphaStable}
It is possible to generalize (\ref{fractionalSDDE}) and Theorem~\ref{existenceTheorem} to allow for a heavy-tailed distribution of the noise. Specifically, suppose that $(L_t)_{t\in \mathbb{R}}$ is a symmetric $\alpha$-stable L\'{e}vy process for some $\alpha \in (1,2)$, that is, $(L_t)_{t\in \mathbb{R}}$ is a L\'{e}vy process and
\begin{align*}
\mathbb{E}\big[e^{iy L_1}\big] = e^{-\sigma^\alpha \vert y \vert^\alpha},\quad y \in \mathbb{R},
\end{align*}
for some $\sigma >0$. To define the process $(X_t)_{t\in \mathbb{R}}$ in (\ref{solution}) it is necessary and sufficient that $g\in L^\alpha$, which is indeed the case if $\beta \in (1,1-1/\alpha)$ by Proposition~\ref{kernels}. From this point, using (\ref{functionalEquation}), we only need a stochastic Fubini result (which can be found in \cite[Theorem~3.1]{QOU}) to verify that (\ref{fractionalSDDE}) is satisfied. One will need another notion (and proof) of uniqueness, however, as our approach relies on $L^2$ theory. For more on stable distributions and corresponding definitions and results, we refer to \cite{Stable}.
\end{remark}

\begin{remark}\label{fractionalModels}
The process (\ref{solution}) and other well-known long-memory processes do naturally share parts of their construction. For instance, they are typically viewed as "borderline"\ stationary solutions to certain equations. To be more concrete, the ARFIMA process can be viewed as an ARMA process, but where the autoregressive polynomial $P$ is replaced by $\tilde{P}:z \mapsto P(z)(1-z)^\beta$. Although an ordinary ARMA process exists if and only if $P$ is non-zero on the unit circle (and, in the positive case, will be a short memory process), the autoregressive function $\tilde{P}$ of the ARFIMA model will always have a root at $z=1$. The analogue to the autoregressive polynomial in the non-fractional SDDE model (that is, (\ref{fractionalSDDE}) with $D^\beta_- \mathds{1}_{(s,t]}$ replaced by $\mathds{1}_{(s,t]}$) is
\begin{align}\label{hBar}
 z \mapsto -z - \mathcal{L}[\eta](z),
\end{align}
where the critical region is on the imaginary axis $\{iy \, :\, y \in \mathbb{R}\}$ rather than on the unit circle $\{z \in \mathbb{C}\, :\, \vert z \vert = 1\}$ (see \cite{contARMAframework}). The SFDDE corresponds to replacing (\ref{hBar}) by $z \mapsto -z -(-z)^\beta \mathcal{L}[\eta](z)$, which will always have a root at $z=0$. However, to ensure existence both in the ARFIMA model and in the SFDDE model, assumptions are made such that these roots will be the only ones in the critical region and their order will be $\beta$. For a treatment of ARFIMA processes, we refer to \cite[Section~13.2]{BrockDavis}.
\end{remark}

The solution $(X_t)_{t\in \mathbb{R}}$ of Theorem~\ref{existenceTheorem} is causal in the sense that $X_t$ only depends on past increments of the noise $L_t-L_s$, $s\leq t$. An inspection of the proof of Theorem~\ref{existenceTheorem} reveals that one only needs to require that $h(iy) \neq 0$ for all $y\in \mathbb{R}$ for a (possibly non-causal) stationary solution to exist. The difference between the condition that $h(z)$ is non-zero when $\text{Re}(z) = 0$ rather than when $\text{Re}(z)\leq 0$ in terms of causality is similar to that of non-fractional SDDEs (see, e.g., \cite{contARMAframework}).

The next result shows why one may view (\ref{fractionalSDDE}) as (\ref{fractionalSDDEcompact}). In particular, it reveals that the corresponding solution $(X_t)_{t\in \mathbb{R}}$ is a semimartingale with respect to (the completion of) its own filtration or equivalently, in light of (\ref{fractionalSDDE}) and (\ref{solution}), the one generated from the increments of $(L_t)_{t\in \mathbb{R}}$.
\begin{proposition}\label{fractionalX}
Suppose that $h(z)$ is non-zero for every $z\in \mathbb{C}$ with $\text{Re}(z)\leq 0$ and let $(X_t)_{t\in \mathbb{R}}$ be the solution to (\ref{fractionalSDDE}) given in Theorem~\ref{existenceTheorem}. Then, for $t\in \mathbb{R}$, the limit
\begin{align}\label{dAlphaX}
D^\beta X_t :=  \frac{\beta}{\Gamma (1-\beta)} \lim_{\delta \downarrow 0}\int_\delta^\infty \frac{X_t-X_{t-u}}{u^{1+\beta}}\, du
\end{align}
exists in $L^2(\mathbb{P})$, $D^\beta X_t = \int_{-\infty}^t D^\beta g (t-u)\, dL_u$, and it holds that
\begin{align}\label{fractionalChange}
\begin{aligned}
\MoveEqLeft\frac{1}{\Gamma (1-\beta)}\int_{-\infty}^t\big[(t-u)^{-\beta}-(s-u)_+^{-\beta} \big] \int_{[0,\infty)} X_{u-v}\, \eta (dv)\, du \\
&= \int_s^t \int_{[0,\infty)}D^\beta X_{u-v}\, \eta (dv)\, du
\end{aligned}
\end{align}
almost surely for each $s<t$.
\end{proposition}

We will now provide some properties of the solution $(X_t)_{t\in \mathbb{R}}$ to (\ref{fractionalSDDE}) given in (\ref{solution}). Since the autocovariance function $\gamma_X$ takes the form
\begin{align}\label{autocovariance}
\gamma_X (t)  = \int_\mathbb{R} g(t+u)g(u)\, du, \quad t \in \mathbb{R},
\end{align}
it follows by Plancherel's theorem that $(X_t)_{t\in \mathbb{R}}$ admits a spectral density $f_X$ which is given by
\begin{align}\label{spectralDensity}
f_X (y) = \vert \mathcal{F}[g](y)\vert^2 = \frac{1}{\vert h(iy)\vert^2} \vert y \vert^{-2\beta}, \quad y \in \mathbb{R}.
\end{align}
(See the appendix for a brief recap of the spectral theory.) The following result concerning $\gamma_X$ and $f_X$ shows that solutions to (\ref{fractionalSDDE}) exhibit a long-memory behavior and that the degree of memory can be controlled by $\beta$. 

\begin{proposition}\label{memory} Suppose that $h(z)$ is non-zero for every $z\in \mathbb{C}$ with $\text{Re}(z)\leq 0$ and let $\gamma_X$ and $f_X$ be the functions introduced in (\ref{autocovariance})-(\ref{spectralDensity}). Then it holds that
\begin{align*}
\gamma_X(t) \sim \frac{\Gamma (1-2\beta)}{\Gamma (\beta) \Gamma (1-\beta) \eta ([0,\infty))^2} t^{2\beta -1} 
\quad \text{and}\quad 
f_X(y) \sim \frac{1}{\eta ([0,\infty))^2} \vert y\vert^{-2\beta}
\end{align*}
as $t\to \infty$ and $y\to 0$, respectively. In particular, $\int_\mathbb{R}\vert \gamma_X (t)\vert\, dt = \infty$.
\end{proposition}

While the behavior of $\gamma_X(t)$ as $t\to \infty$ is controlled by $\beta$, the content of Proposition~\ref{roughness} is that the behavior of $\gamma_X (t)$ as $t\to 0$, and thus the $L^2(\mathbb{P})$-H\"{o}lder continuity of the sample paths of $(X_t)_{t\in \mathbb{R}}$ (cf. Remark~\ref{holderRem}), is unaffected by $\beta$.

\begin{proposition}\label{roughness} Suppose that $h (z)$ is non-zero for every $z\in \mathbb{C}$ with $\text{Re}(z)\leq 0$, let $(X_t)_{t\in \mathbb{R}}$ be the solution to (\ref{fractionalSDDE}) and denote by $\rho_X$ its ACF. Then it holds that $1-\rho_X(t) \sim t$ as $t \downarrow 0$.
\end{proposition}

\begin{remark}\label{holderRem}
Recall that for a given $\gamma >0$, a centered and square integrable process $(X_t)_{t\in \mathbb{R}}$ with stationary increments is said to be locally $\gamma$-H\"{o}lder continuous in $L^2(\mathbb{P})$ if there exists a constant $C>0$ such that 
\begin{align*}
\frac{\mathbb{E}\big[(X_t-X_0)^2\big]}{t^{2\gamma}} \leq C
\end{align*}
for all sufficiently small $t>0$. By defining the semi-variogram 
\begin{align*}
\gamma_V (t):= \tfrac{1}{2}\mathbb{E}[(X_t - X_0)^2], \quad t \in \mathbb{R},
\end{align*}
we see that $(X_t)_{t\in \mathbb{R}}$ is locally $\gamma$-H\"{o}lder continuous if and only if $\gamma_V (t) = O(t^{2\gamma})$ as $t \to 0$. When $(X_t)_{t\in \mathbb{R}}$ is stationary we have the relation $\gamma_V = \gamma_X (0) (1-\rho_X )$, from which it follows that the $L^2(\mathbb{P})$ notion of H\"{o}lder continuity can be characterized in terms of the behavior of the ACF at zero. In particular, Proposition~\ref{roughness} shows that the solution $(X_t)_{t\in \mathbb{R}}$ to (\ref{fractionalSDDE}) is locally $\gamma$-H\"{o}lder continuous if and only if $\gamma \leq 1/2$. The behavior of the ACF at zero has been used as a measure of roughness of the sample paths in for example \cite{bennedsen2015rough,bennedsen2016decoupling}. 
\end{remark}

\begin{remark}\label{jumps} As a final comment on the path properties of the solution $(X_t)_{t\in \mathbb{R}}$ to (\ref{fractionalSDDE}), observe that
\begin{align*}
X_t -X_s = \int_s^t \int_{[0,\infty)} D^\beta X_{u-v}\, \eta (dv)\, du  + L_t -L_s
\end{align*}
for each $s<t$ almost surely by Proposition~\ref{fractionalX}. This shows that $(X_t)_{t\in \mathbb{R}}$ can be chosen so that it has jumps at the same time (and of the same size) as $(L_t)_{t\in \mathbb{R}}$. This is in contrast to models driven by a fractional L\'{e}vy process, such as (\ref{usualFracSDDE}), since $(I^\beta L_t)_{t\in \mathbb{R}}$ is continuous in $t$ (see \cite[Theorem~3.4]{Tina}).
\end{remark}

We end this section by providing a formula for computing $\mathbb{E}[X_t\mid X_u,\, u \leq s]$ for any $s<t$. One should compare its form to those obtained for other fractional models (such as the one in \cite[Theorem~3.2]{basse2018multivariate} where, as opposed to Proposition~\ref{prediction}, the prediction is expressed not only in terms of its own past, but also the past noise).

\begin{proposition}\label{prediction}
Suppose that $h(z)$ is non-zero for every $z\in \mathbb{C}$ with $\text{Re}(z)\leq 0$ and let $(X_t)_{t\in \mathbb{R}}$ denote the solution to (\ref{fractionalSDDE}). Then for any $s<t$, it holds that
\begin{align*}
\begin{aligned}
\MoveEqLeft\mathbb{E}[X_t\mid X_u,\, u\leq s] = g(t-s)X_s\\
&+\int_{[0,t-s)} \int_{-\infty}^s X_w  \int_{[0,\infty)}\big(D^\beta_-\mathds{1}_{(s,t-u]}\big)(v+w)  \,\eta(dv) \, dw \, g(du),
\end{aligned}
\end{align*}
where $g (du) = \delta_0 (du) + (D^\beta g)\ast \eta (u)\, du$ is the Lebesgue-Stieltjes measure induced by $g$.
\end{proposition}

\section{Delays of exponential type}\label{expSection}
Let $A$ be an $n\times n$ matrix where all its eigenvalues belong to $\{z\in \mathbb{C}\, :\, \text{Re}(z)<0\}$, and let $b\in \mathbb{R}^n$ and $\kappa\in \mathbb{R}$. In this section we restrict our attention to measures $\eta$ of the form
\begin{align}\label{exponentialEta}
\eta (du) = -\kappa \delta_0 (du) + f(u)\, du, \quad \text{with}\quad  f(u) = b^T e^{Au}e_1,
\end{align}
where $e_1 := (1,0,\dots, 0)^T \in \mathbb{R}^n$. Note that $e_1$ is used as a normalization; the effect of replacing $e_1$ by any $c\in \mathbb{R}^n$ can be incorporated in the choice of $A$ and $b$. It is well-known that the assumption on the eigenvalues of $A$ imply that all the entries of $e^{Au}$ decay exponentially fast as $u\to \infty$, so that $\eta$ is a finite measure on $[0,\infty)$ with moments of any order. Since the Fourier transform $\mathcal{F}[f]$ of $f$ is given by 
\begin{align*}
\mathcal{F}[f](y) = -b^T(A+iyI_n)^{-1}e_1,\quad y\in \mathbb{R},
\end{align*}
it admits a fraction decomposition; that is, there exist real polynomials $Q,R: \mathbb{C}\to \mathbb{C}$, $Q$ being monic with the eigenvalues of $A$ as its roots and being of larger degree than $R$, such that
\begin{align}\label{RQ-polynomials}
\mathcal{F}[f](y) = -\frac{R(-iy)}{Q(-iy)}
\end{align}
for $y\in \mathbb{R}$. (This is a direct consequence of the inversion formula $B^{-1} = \text{adj} (B)/\det (B)$.) By assuming that $Q$ and $R$ have no common roots, the pair $(Q,R)$ is unique. The following existence and uniqueness result is simply an application of Theorem~\ref{existenceTheorem} to the particular setup in question:
\begin{corollary}\label{exponentialExistence} Let $Q$ and $R$ be given as in (\ref{RQ-polynomials}). Suppose that $\kappa + b^TA^{-1}e_1 \neq 0$ and
\begin{align}\label{ConditionExpExist}
Q(z)\big[z+\kappa z^\beta \big] + R(z)z^\beta \neq 0
\end{align}
for all $z \in \mathbb{C}\setminus \{0\}$ with $\text{Re}(z) \geq 0$. Then there exists a unique purely non-deterministic solution $(X_t)_{t\in \mathbb{R}}$ to (\ref{fractionalSDDE}) with $\eta$ given by (\ref{exponentialEta}) and it is given by (\ref{solution}) with $g:\mathbb{R}\to \mathbb{R}$ characterized through the relation
\begin{align}\label{gExponential}
\mathcal{F}[g](y) = \frac{Q(-iy)}{Q(-iy)\big[-iy+\kappa (-iy)^\beta\big] + R(-iy)(-iy)^\beta},\quad y \in \mathbb{R}.
\end{align}
\end{corollary}

Before giving examples we state Proposition~\ref{fracDiffChange}, which shows that the general SFDDE (\ref{fractionalSDDE}) can be written as
\begin{align}\label{exponentialSDDE}
d X_t =-\kappa D^\beta X_t\, dt +\int_0^\infty X_{t-u} D^\beta f (u)\, du\, dt + dL_t,\quad t \in \mathbb{R},
\end{align}
when $\eta$ is of the form (\ref{exponentialEta}). In case $\kappa =0$,  (\ref{exponentialSDDE}) is a (non-fractional) SDDE. However, the usual existence results obtained in this setting (for instance, those in \cite{contARMAframework} and \cite{GK}) are not applicable, since the delay measure $D^\beta f(u)\, du$ has unbounded support and zero total mass $\int_0^\infty D^\beta f(u)\, du = 0$.

\begin{proposition}\label{fracDiffChange} Let $f$ be of the form (\ref{exponentialEta}). Then $D^\beta f:\mathbb{R}\to \mathbb{R}$ defined by $D^\beta f(t) = 0$ for $t\leq 0$ and
\begin{align*}
D^\beta f (t) = \frac{1}{\Gamma (1-\beta)}b^T\biggr( Ae^{At}\int_0^t e^{-Au} u^{-\beta}\, du +  t^{-\beta}I_n \biggr)e_1
\end{align*}
for $t>0$ belongs to $L^1 \cap L^2$. If in addition, (\ref{ConditionExpExist}) holds, $\kappa + b^TA^{-1}e_1 \neq 0$, and $(X_t)_{t\in \mathbb{R}}$ is the solution given in Corollary~\ref{exponentialExistence}, then
\begin{align*}
\int_0^\infty D^\beta X_{t-u} f(u)\, du = \int_0^\infty X_{t-u}\, D^\beta f (u)\, du
\end{align*}
almost surely for any $t\in \mathbb{R}$.
\end{proposition}

\begin{remark}\label{RelationToFICARMA}
Due to the structure of the function $g$ in (\ref{gExponential}) one may, in line with the interpretation of CARMA processes, think of the corresponding solution $(X_t)_{t\in \mathbb{R}}$ as a stationary process that satisfies the formal equation
\begin{align}\label{CARMAinterpret}
\big( Q(D)\big[D+\kappa D^\beta\big] + R(D)D^\beta \big) X_t = Q(D)DL_t,\quad t \in \mathbb{R},
\end{align}
where $D$ denotes differentiation with respect to $t$ and $D^\beta$ is a suitable fractional derivative. Indeed, by heuristically applying the Fourier transform $\mathcal{F}$ to (\ref{CARMAinterpret}) and using computation rules such as $\mathcal{F}[DX](y) = (-iy)\mathcal{F}[X](y)$ and $\mathcal{F}[D^\beta X](y) = (-iy)^\beta \mathcal{F}[X](y)$, one ends up concluding that $(X_t)_{t\in \mathbb{R}}$ is of the form (\ref{solution}) with $g$ characterized by (\ref{gExponential}). For two monic polynomials $P$ and $Q$ with $q:=\text{deg}(Q)=\text{deg}(P) -1$ and all their roots contained in $\{z \in \mathbb{C}\, :\, \text{Re}(z) <0\}$, consider the FICARMA($q+1,\beta,q$) process $(X_t)_{t\in \mathbb{R}}$. Heuristically, by applying $\mathcal{F}$ as above, $(X_t)_{t\in \mathbb{R}}$ may be thought of as the solution to $P(D)D^\beta X_t = Q(D)DL_t$, $t\in \mathbb{R}$. By choosing the polynomial $R$ and the constant $\kappa$ such that $P(z) = Q(z)[z+\kappa] + R(z)$ we can think of $(X_t)_{t\in \mathbb{R}}$ as the solution to the formal equation
\begin{align}\label{CARMAinterpret2}
\big(Q(D)\big[D^{1+\beta}+\kappa D^\beta\big] + R(D)D^\beta \big) X_t = Q(D)DL_t,\quad t \in \mathbb{R}.
\end{align} 
It follows that (\ref{CARMAinterpret}) and (\ref{CARMAinterpret2}) are closely related, the only difference being that $D+\kappa D^\beta$ is replaced by $D^{1+\beta}+ \kappa D^\beta$. In particular, one may view solutions to SFDDEs corresponding to measures of the form (\ref{exponentialEta}) as being of the same type as FICARMA processes. While the considerations above apply only to the case where $\text{deg}(P)=q+1$, it should be possible to extend the SFDDE framework so that solutions are comparable to the FICARMA processes in the general case $\text{deg}(P)>q$ by following the lines of \cite{basse2018multivariate}, where similar theory is developed for the SDDE setting.
\end{remark}

We will now give two examples of (\ref{exponentialSDDE}).

\begin{example}\label{OU-example}
Consider choosing $\eta= -\kappa\delta_0$ for some $\kappa >0$ so that (\ref{fractionalSDDE}) becomes
\begin{align}\label{OUsimeqn}
X_t - X_s = -\frac{\kappa}{\Gamma (1-\beta)} \int_{-\infty}^t \big[(t-u)^{-\beta}-(s-u)_+^{-\beta} \big]\, X_u\, du + L_t - L_s
\end{align}
for $s<t$ or, in short,
\begin{align}\label{OUequationShort}
dX_t = -\kappa D^\beta X_t\, dt + dL_t,\quad t \in\mathbb{R}.
\end{align}

To argue that a unique purely non-deterministic solution exists, we observe that $Q(z)=1$ and $R(z) = 0$ for all $z \in \mathbb{C}$. Thus, in light of Corollary~\ref{exponentialExistence} and (\ref{ConditionExpExist}), it suffices to argue that $z + \kappa z^\beta \neq 0$ for all $z \in \mathbb{C}\setminus \{0\}$ with $\text{Re}(z)\geq 0$. By writing such $z$ as $z =r e^{i\theta}$ for a suitable $r>0$ and $\theta \in [-\pi/2,\pi/2]$, the condition may be written as
\begin{align}\label{OU-Frac}
\big(r\cos (\theta) + \kappa r^\beta\cos (\beta \theta) \big) +i \big(r\sin (\theta) + \kappa r^\beta \sin (\beta\theta) \big) \neq 0.
\end{align}
If the imaginary part of the left-hand side of (\ref{OU-Frac}) is zero it must be the case that $\theta =0$, since $\kappa >0$ while $\sin (\theta)$ and $\sin (\beta\theta)$ are of the same sign. However, if $\theta = 0$, the real part of the left-hand side of (\ref{OU-Frac}) is $r+\kappa r^\beta >0$. Consequently, Corollary~\ref{exponentialExistence} implies that a solution to (\ref{OUequationShort}) is characterized by (\ref{solution}) and $\mathcal{F}[g](y) = ((-iy)^\beta \kappa - iy)^{-1}$ for $y\in \mathbb{R}$. In particular, $\gamma_X$ takes the form
\begin{align}\label{ACF-OU}
\gamma_X (t) &= \int_\mathbb{R} \frac{e^{ity}}{y^2 + 2\kappa \sin (\tfrac{\beta \pi}{2})\vert y \vert^{1+ \beta}+ \kappa^2 \vert y \vert^{2\beta}}\, dy.
\end{align}
In Figure~\ref{plots-OU} we have plotted the ACF of $(X_t)_{t\in \mathbb{R}}$ using (\ref{ACF-OU}) with $\kappa = 1$ and $\beta \in \{0.1,0.2,0.3,0.4\}$. We compare it to the ACF of the corresponding fractional Ornstein-Uhlenbeck process (equivalently, the FICARMA($1,\beta,0$) process) which was presented in (\ref{fOUcompact}). To do so, we use that its autocovariance function $\gamma_\beta$ is given by
\begin{align}\label{ACF-FICAR}
\gamma_\beta(t) =  \int_\mathbb{R} \frac{e^{ity}}{\vert y \vert^{2(1+\beta)} + \kappa^2 \vert y \vert^{2\beta}}\, dy.
\end{align}
From these plots it becomes evident that, although the ACFs share the same behavior at infinity, they behave differently near zero. In particular, we see that the ACF of $(X_t)_{t \in \mathbb{R}}$ decays more rapidly around zero, which is in line with Proposition~\ref{roughness} and the fact that the $L^2(\mathbb{P})$-H\"{o}lder continuity of the fractional Ornstein-Uhlenbeck process increases as $\beta$ increases (cf. the introduction).

\begin{figure}[h]
\centering
\includegraphics[scale=0.25]{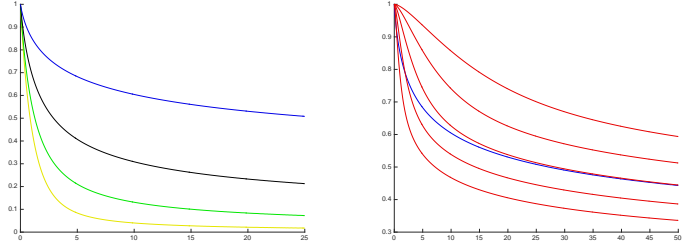}
\caption{The left plot is the ACF based on (\ref{ACF-OU}) with $\beta = 0.1$ (yellow), $\beta = 0.2$ (green), $\beta = 0.3$ (black) and $\beta = 0.4$ (blue). With $\beta = 0.4$ fixed, the plot on the right compares the ACF based on (\ref{ACF-OU}) with $\kappa = 1$ (blue) to the ACF based on (\ref{ACF-FICAR}) for $\kappa =0.125,0.25,0.5,1,2$ (red) where the ACF decreases in $\kappa$, in particular, the top curve corresponds to $\kappa = 0.125$ and the bottom to $\kappa = 2$.}\label{plots-OU}
\end{figure}
\end{example}

\begin{example}\label{simpleExp} Suppose that $\eta$ is is given by (\ref{exponentialEta}) with $\kappa = 0$, $A = -\kappa_1$, and $b=-\kappa_2$ for some $\kappa_1,\kappa_2 >0$. In this case, $f(t) = -\kappa_2 e^{-\kappa_1 t}$ and (\ref{exponentialSDDE}) becomes
\begin{align}\label{simpleExpExample}
dX_t = \frac{\kappa_2}{\Gamma (1-\beta)}\int_0^\infty X_{t-u}\biggr(\kappa_1 e^{-\kappa_1 u}\int_0^u e^{\kappa_1v} v^{-\beta}\, dv - u^{-\beta} \biggr)\, du\, dt + dL_t,
\end{align}
and since $Q(z) = z+ \kappa_1$ and $R(z) = \kappa_2$ we have that
\begin{align*}
z Q(z) + R(z)z^\beta = z^2 + \kappa_1 z + \kappa_2 z^\beta.
\end{align*}
To verify (\ref{ConditionExpExist}), set $z = x+iy$ for $x>0$ and $y\in \mathbb{R}$ and note that
\begin{align}\label{exampleIdentity}
\begin{aligned}
z^2 + \kappa_1 z + \kappa_2 z^\beta =& \big(x^2-y^2 + \kappa_1 x + \kappa_2 \cos (\beta \theta_z) \vert z\vert^\beta \big) \\
&+ i\big(\kappa_1y +2xy + \kappa_2\sin (\beta\theta_z)\vert z \vert^\beta \big)
\end{aligned}
\end{align}
for a suitable $\theta_z \in (-\pi/2,\pi/2)$. For the imaginary part of (\ref{exampleIdentity}) to be zero it must be the case that
\begin{align*}
(\kappa_1+2x)y = - \kappa_2 \sin (\beta \theta_z)\vert z \vert^\beta,
\end{align*}
and this can only happen if $y=0$, since $x,\kappa_1,\kappa_2>0$ and the sign of $y$ is the same as that of $\sin (\beta \theta_z)$. However, if $y=0$ it is easy to see that the real part of (\ref{exampleIdentity}) cannot be zero for any $x>0$, so we conclude that (\ref{ConditionExpExist}) holds and that there exists a stationary solution $(X_t)_{t\in \mathbb{R}}$ given through the kernel (\ref{gExponential}). The autocovariance function $\gamma_X$ is given by
\begin{align}\label{simpleExAutoCov}
\gamma_X(t) &= \int_\mathbb{R}e^{ity}\frac{y^2+ \kappa_1^2}{y^4 + 2\kappa_2\big(\kappa_1\gamma_2\vert y \vert^{1+\beta}-\gamma_1\vert y \vert^{2+\beta} \big) + \kappa_1^2y^2+\kappa_2^2 \vert y \vert^{2\beta}}\, dy
\end{align}
where $\gamma_1 = \cos (\beta \pi /2)$ and $\gamma_2 = \sin (\beta \pi /2)$. The polynomials to the associated FICARMA($2, \beta,1$) process are given by $P(z) = z^2 + \kappa_1 z + \kappa_2$ and $Q(z) = z + \kappa_1$ (see Remark~\ref{RelationToFICARMA}) and the autocovariance function $\gamma_\beta$ takes the form
\begin{align}\label{ACFficarma}
\gamma_\beta(t) = \int_\mathbb{R}e^{ity}\frac{y^2+ \kappa_1^2}{\vert y\vert^{4+2\beta} +(\kappa_1^2-2\kappa_2)\vert y\vert^{2+2\beta} + \kappa_2^2\vert y \vert^{2\beta}}\, dy.
\end{align}
In Figure~\ref{simpleExACF} we have plotted the ACF based on (\ref{simpleExAutoCov}) for $\kappa_1 = 1$ and various values of $\kappa_2$ and $\beta$. For comparison we have also plotted the ACF based on (\ref{ACFficarma}) for the same choices of $\kappa_1$, $\kappa_2$ and $\beta$.

\begin{figure}[h]
\centering
\includegraphics[scale=0.27]{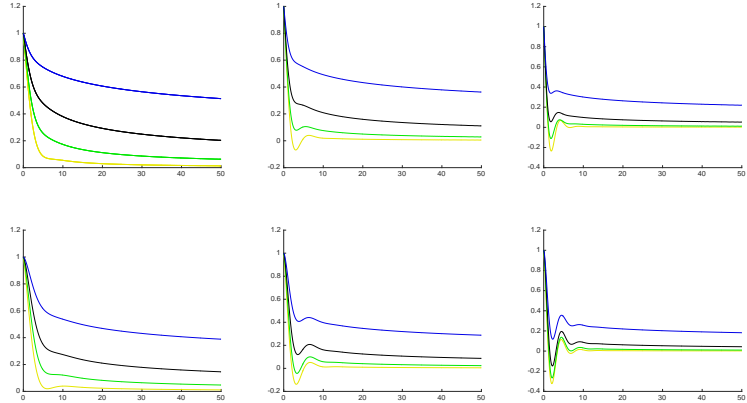}
\caption{First row is ACF based on (\ref{simpleExAutoCov}), second row is ACF based on (\ref{ACFficarma}), and the columns correspond to $\kappa_2 = 0.5$, $\kappa_2 = 1$ and $\kappa_2 = 2$, respectively. Within each plot, the lines correspond to $\beta =0.1$ (yellow), $\beta = 0.2$ (green), $\beta = 0.3$ (black) and $\beta = 0.4$ (blue). In all plots, $\kappa_1 = 1$.}\label{simpleExACF}
\end{figure}
\end{example}

\section{Simulation from the SFDDE}\label{estimation}
%, since \eqref{fractionalSDDEcompact} is not well-suited for numerical use.The reason is that the operator $D^\beta$ is unstable when approximation errors are introduced. On the other hand, the equation in \eqref{fractionalSDDE} is considerably more stable, since we move $D^\beta$ to the indicator function $\mathds{1}_{(s,t]}$ and $D^\beta_-\mathds{1}_{(s,t]}$ can be determined explicitly.

In the following we will focus on simulating from (\ref{fractionalSDDE}). We begin this simulation study by considering the Ornstein-Uhlenbeck type equation discussed in Example \ref{OU-example} with $\kappa=1$ and under the assumption that $(L_t)_{t\in \mathbb{R}}$ is a standard Brownian motion. Let $c_1 =100/\Delta$ and $c_2 = 2000/\Delta$. We generate a simulation of the solution process $(X_t)_{t \in \RR}$ on a grid of size $\Delta = 0.01$ and with $3700/\Delta$ steps of size $\Delta$ starting from $-c_1-c_2$ and ending at $1600/\Delta$. Initially, we set $X_t$ equal to zero for the first $c_1$ points in the grid and then discretize (\ref{OUsimeqn}) using the approximation
\begin{align*}
 \MoveEqLeft \int_\mathbb{R} \big[(n\Delta-u)_+^{-\beta}-((n-1)\Delta-u)_+^{-\beta} \big]\, X_u\, du \\
 \simeq  &\frac{1}{1-\beta}\Delta^{1-\beta}X_{(n-1)\Delta}   \\
 &+ \sum_{k=n-c_1}^{n-1} \frac{X_{k\Delta} +X_{(k-1)\Delta}}{2} \int_{(k-1)\Delta}^{k\Delta}\big[ (n\Delta-u)_+^{-\beta}-((n-1)\Delta-u)_+^{-\beta}\big] \, du \\
 =&  \frac{1}{1-\beta}\Delta^{1-\beta}X_{(n-1)\Delta}   + \frac{1}{1-\beta} \sum_{k=n-c_1}^{n-1} \frac{X_{k\Delta} +X_{(k-1)\Delta}}{2}\\
 & \cdot \big( 2((n-k-1)\Delta)^{1-\beta} -((n-k)\Delta)^{1-\beta} -((n-k-2)\Delta)^{1-\beta}\big)
\end{align*} 
for $n=-c_2+1,\dots, 3700/\Delta-c_2-c_1$. Next, we disregard the first $c_1 +c_2$ values of the simulated sample path to obtain an approximate sample from the stationary distribution. We assume that the process is observed on a unit grid resulting in simulated values $X_1,\dots,X_{1600}$. This is repeated $200$ times, and in every repetition the sample ACF based on $X_1,\dots,X_L$ is computed for $t=1,\dots, 25$ and $L =100,400,1600$. In long-memory models, the sample mean $\bar{X}_L$ can be a poor approximation to the true mean $\mathbb{E}[X_0]$ even for large $L$, and this may result in considerable negative (finite sample) bias in the sample ACF (see, e.g., \cite{newbold1993bias}). Due to this bias, it may be difficult to see if we succeed in simulating from (\ref{fractionalSDDE}), and hence we will assume that $\mathbb{E}[X_0]$ is known to be zero when computing the sample ACF. We calculate the $95\%$ confidence interval
\begin{align*}
\left[ \bar{\rho}(k) - 1.96 \tfrac{\hat{\sigma}(k)}{\sqrt{200}},  \bar{\rho}(k) + 1.96 \tfrac{\hat{\sigma}(k)}{\sqrt{200}}\right], 
\end{align*} 
for the mean of the sample ACF based on $L$ observations at lag $k$. Here $\bar{\rho}(k)$ is the sample mean and $\hat{\sigma}(k)$ is the sample standard deviations of the ACF at lag $k$ based on the $200$ replications. In Figure~\ref{autocovsim1}, the theoretical ACFs and the corresponding $95\%$ confidence intervals for the mean of the sample ACFs are plotted for $\beta =0.1,0.2$ and $L=100,400,1600$. We see that, when correcting for the bias induced by an unknown mean $\mathbb{E}[X_0]$, simulation from equation (\ref{OUsimeqn}) results in a fairly unbiased estimator of the ACF for small values of $\beta$. When $\beta>0.25$, in the case where the ACF of $(X_t)_{t \in \RR}$ is not even in $L^2$, the results are more unstable as it requires large values of $c_1$ and $c_2$ to ensure that the simulation results in a good approximation to the stationary distribution of $(X_t)_{t\in \mathbb{R}}$. Moreover, even after correcting for the bias induced by an unknown mean of the observed process, the sample ACF for the ARFIMA process shows considerable finite sample bias when $\beta>0.25$, see \cite{newbold1993bias}, and hence we may expect this to apply to solutions to (\ref{fractionalSDDE}) as well.

\begin{figure}[h]
\centering
\includegraphics[scale=0.4]{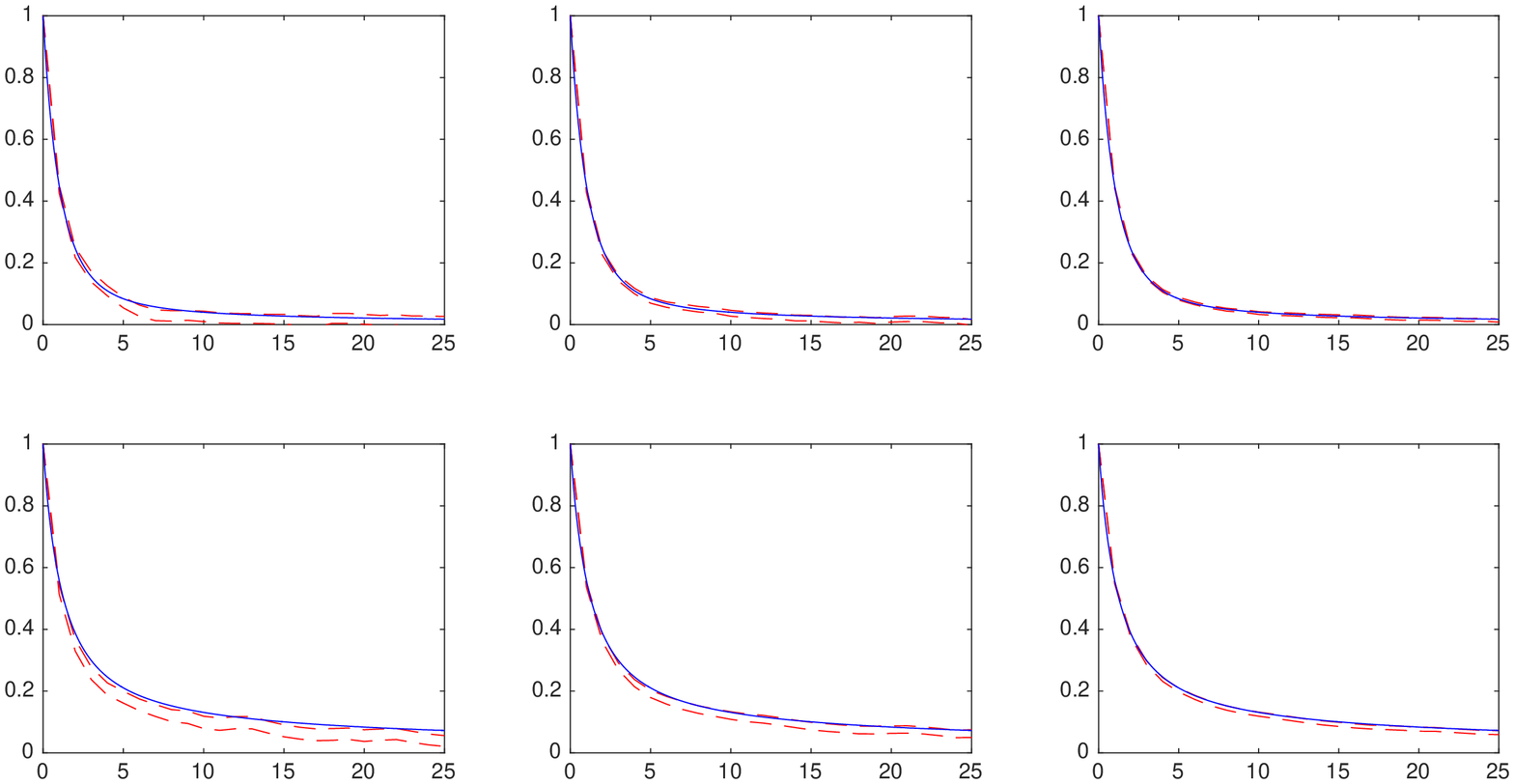}
\caption{Theoretical ACF and $95\%$ confidence intervals of the mean of the sample ACF based on $200$ replications of $X_1,\dots, X_L$. Columns correspond to $L=100$, $L=400$ and $L= 1600$, respectively, and rows correspond to $\beta=0.1$ and $\beta = 0.2$, respectively. The model is (\ref{OUsimeqn}).}\label{autocovsim1}
\end{figure}

In Figure~\ref{autocovsim2} we have plotted box plots for the $200$ replications of the sample ACF for $\beta=0.1,0.2$ and $L=100,400,1600$. We see that the sample ACFs have the expected convergence when $L$ grows and that the distribution is more concentrated in the case where less memory is present. 

\begin{figure}[h]
\centering
\includegraphics[scale=0.475]{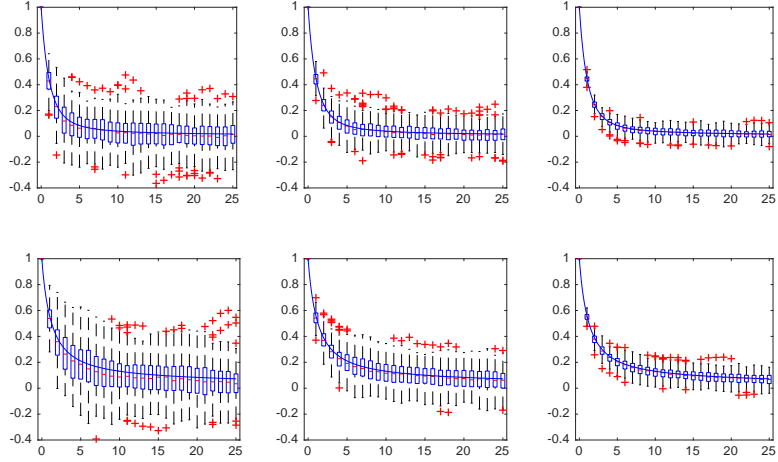}
\caption{Box plots for the sample ACF based on $200$ replications of $X_1,\dots,X_L$ together with the theoretical ACF. Columns correspond to $L=100$, $L=400$ and $L= 1600$, respectively, and rows correspond to $\beta=0.1$ and $\beta = 0.2$, respectively. The model is (\ref{OUsimeqn}).}\label{autocovsim2}
\end{figure}

Following the same approach as above, we simulate the solution to the equation discussed in Example~\ref{simpleExp}. Specifically, the simulation is based on equation (\ref{fractionalSDDE}), restricted to the case where $\eta(dv) = -e^{-v}\, dv$ and $(L_t)_{t \in \RR}$ is a standard Brownian motion. In this case, we use the approximation
\begin{align*}
 \MoveEqLeft \int_\mathbb{R} \big[(n\Delta-u)_+^{-\beta}-((n-1)\Delta-u)_+^{-\beta} \big]\, \int_0^\infty X_{u-v} \, e^{-v} \,dv \, du \\
  =& \int_0^\infty X_{n\Delta -v} \int_0^v \big[(u- \Delta)_+^{-\beta}-u_+^{-\beta} \big]e^{u-v} \, du \, dv \\
  \simeq & \frac{1}{2} \Delta X_{(n-1)\Delta} f(\Delta) \\
 &+  \sum_{k=2}^{c_1} \frac{1}{4}\Delta (X_{(n-k)\Delta} + X_{(n-k+1)\Delta})(\varphi(k\Delta) + \varphi((k-1)\Delta))
\end{align*}
where $\varphi : \RR \to \RR$ is given by 
\begin{align*}
\varphi(v) = \int_0^v \big[(u- \Delta)_+^{-\beta}-u^{-\beta} \big]e^{u-v} \, dv.
\end{align*}
We approximate $\varphi$ recursively by noting that
\begin{align*}
\varphi(k \Delta  ) &=  \int_0^{k \Delta } \big[(u- \Delta)_+^{-\beta}-u^{-\beta} \big]e^{u-k \Delta } \, dv \\
&\simeq \frac{1+e^{-\Delta}}{2} \int_{(k-1)\Delta}^{k\Delta}  \big[(u- \Delta)_+^{-\beta}-u_+^{-\beta} \big] \, dv + e^{-\Delta} \varphi((k-1)\Delta)\\
& =  \frac{1}{1-\beta}\frac{1+e^{-\Delta}}{2}\big[((k-1)\Delta)^{1-\beta} - (k\Delta)^{1-\beta}) \big] + e^{-\Delta} \varphi((k-1)\Delta)
\end{align*}
for $k\geq 1$. The theoretical ACFs and corresponding $95\%$ confidence intervals are plotted in Figure~\ref{autocovsim3} and the box plots in Figure~\ref{autocovsim4}. The findings are consistent with first example that we considered.
\begin{figure}[h]
\centering
\includegraphics[scale=0.4]{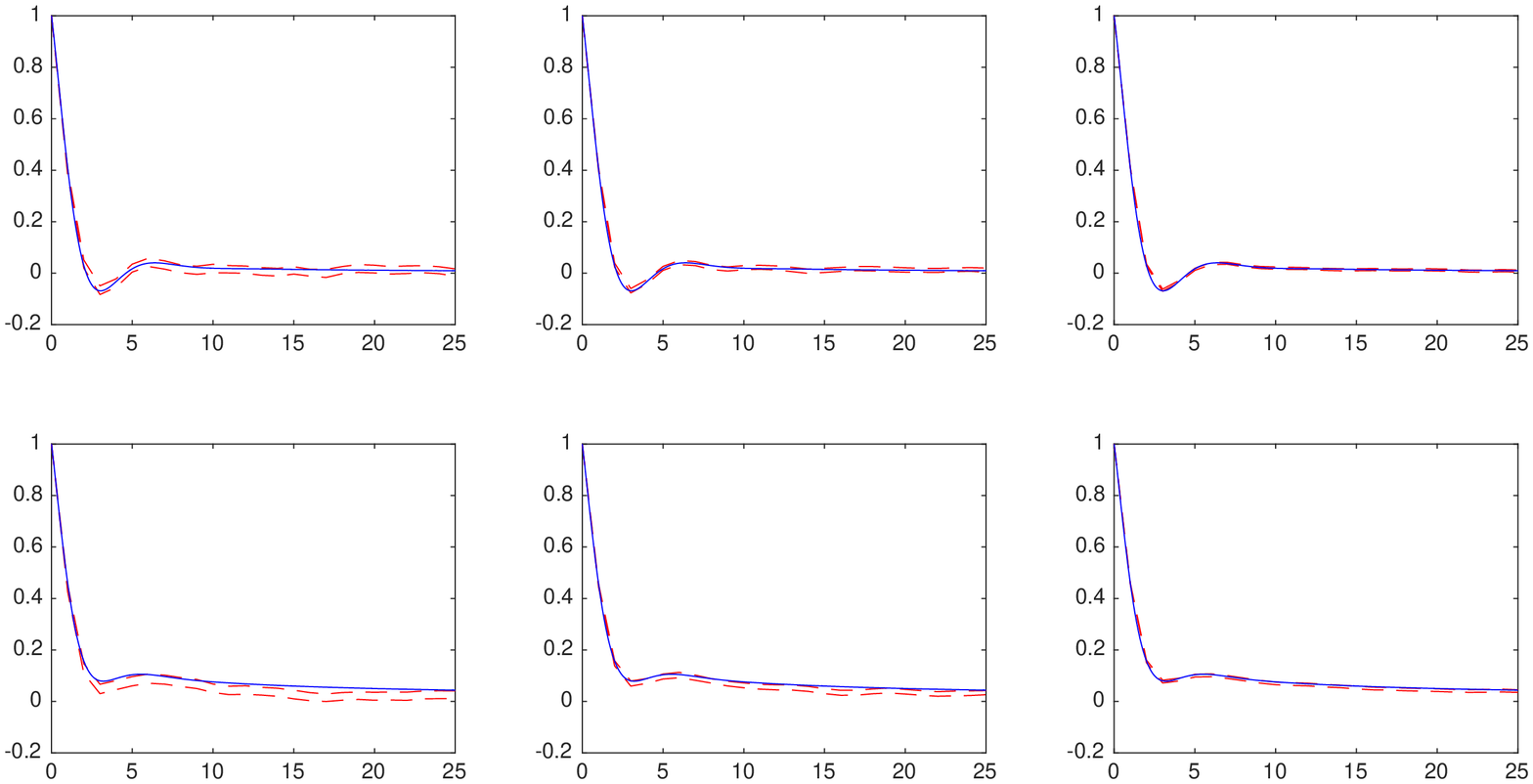}
\caption{Theoretical ACF and $95\%$ confidence intervals of the mean of the sample ACF sample based on $200$ replications of $X_1,\dots, X_L$. Columns correspond to $L=100$, $L=400$ and $L= 1600$, respectively, and rows correspond to $\beta=0.1$ and $\beta = 0.2$, respectively. The model is (\ref{simpleExpExample}).}\label{autocovsim3}
\end{figure}
\begin{figure}[h]
\centering
\includegraphics[scale=0.475]{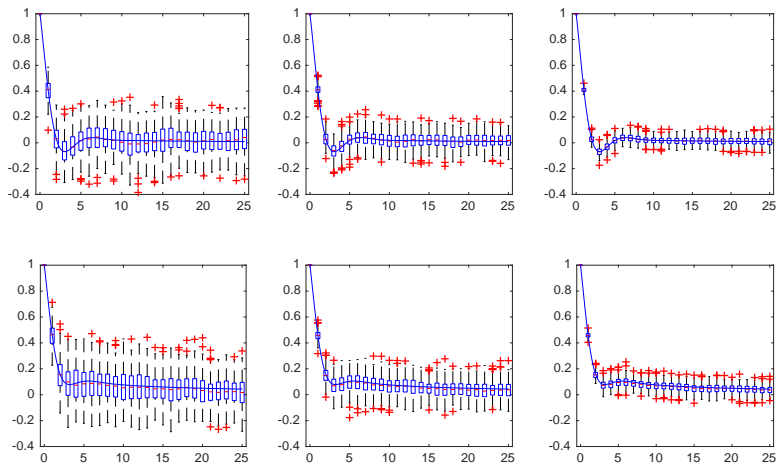}
\caption{Box plots for the sample ACF based on $200$ replications of $X_1,\dots, X_L$ together with the theoretical ACF. Columns correspond to $L=100$, $L=400$ and $L= 1600$, respectively, and rows correspond to $\beta=0.1$ and $\beta = 0.2$, respectively. The model is (\ref{simpleExpExample}).}\label{autocovsim4}
\end{figure}

\FloatBarrier

\appendix

\section{Spectral representations of continuous-time stationary processes}
This appendix provides an exposition of the spectral representation for continuous-time stationary, centered and square integrable processes with a continuous autocovariance function. The proofs are found in Appendix~\ref{proofs} For an extensive treatment we refer to \cite[Section~9.4]{grimmett2001probability} and \cite[Appendix~A2.1]{koopmanSpectral}.

Recall that if $S = \{S(t)\, :\, t \in \mathbb{R}\}$ is a (complex-valued) process such that
\begin{enumerate}[(i)]
\item $\mathbb{E}[\vert S(t) \vert^2]< \infty$ for all $t\in \mathbb{R}$,
\item $\mathbb{E}[\vert S(t+s)-S(t)\vert^2]\to 0$ as $s\downarrow 0$ for all $t\in \mathbb{R}$, and
\item $\mathbb{E}[(S(v)-S(u))\overline{(S(t)-S(s))}] = 0$ for all $u \leq v \leq s \leq t$,
\end{enumerate}
we may (and do) define integration of $f$ with respect to $S$ in the sense of \cite[pp 388-390]{grimmett2001probability} for any $f \in L^2(G)$, where $G$ is the control measure characterized by
\begin{align*}
G((s,t]) = \mathbb{E}\big[\vert S(t) - S(s)\vert^2\big]
\end{align*}
for $s\leq t$. We have the following stochastic Fubini result for this type of integral:
\begin{proposition}\label{stochFubini} Let $S = \{S(t)\, :\, t \in \mathbb{R}\}$ be a process given as above. Let $\mu$ be a finite Borel measure on $\mathbb{R}$, and let $f:\mathbb{R}^2\to \mathbb{C}$ be a measurable function in $L^2(\mu \times G)$. Then all the integrals below are well-defined and
\begin{align}\label{FubiniRelation}
\int_\mathbb{R} \biggr(\int_\mathbb{R} f(x,y)\, \mu (dx) \biggr) \, S(dy) = \int_\mathbb{R} \biggr(\int_\mathbb{R} f(x,y)\, S(dy) \biggr)\, \mu (dx) 
\end{align}
almost surely.
\end{proposition}

Suppose that $(X_t)_{t\in \mathbb{R}}$ is a stationary process with $\mathbb{E}[X_0^2]< \infty$ and $\mathbb{E}[X_0] = 0$, and denote by $\gamma_X $ its autocovariance function. Assuming that $\gamma_X$ is continuous, it follows by Bochner's theorem that there exists a finite Borel measure $F_X$ on $\mathbb{R}$ having $\gamma_X$ as its Fourier transform, that is,
\begin{align*}
\gamma_X (t) = \int_\mathbb{R}e^{ity}F_X(dy) ,\quad t \in \mathbb{R}.
\end{align*}
The measure $F_X$ is referred to as the spectral distribution of $(X_t)_{t\in \mathbb{R}}$.

\begin{theorem}\label{spectralTheorem} Let $(X_t)_{t\in \mathbb{R}}$ be given as above and let $F_X$ be the associated spectral distribution. Then there exists a (complex-valued) process $\Lambda_X = \{\Lambda_X(y)\, :\, y\in \mathbb{R}\}$ satisfying (i)-(iii) above with control measure $F_X$, such that
\begin{align}\label{specRepresentation}
X_t = \int_\mathbb{R}e^{ity}\, \Lambda_X (dy)
\end{align}
almost surely for each $t\in \mathbb{R}$. The process $\Lambda_X$ is called the spectral process of $(X_t)_{t\in \mathbb{R}}$ and (\ref{specRepresentation}) is referred to as its spectral representation.
\end{theorem}

\begin{remark}
Let the situation be as in Theorem~\ref{spectralTheorem} and note that if there exists another process $\tilde{\Lambda}_X = \{\tilde{\Lambda}_X(y)\, :\, y\in \mathbb{R}\}$ such that
\begin{align*}
X_t = \int_\mathbb{R}e^{ity}\, \tilde{\Lambda}_X (dy)
\end{align*}
for all $t \in \mathbb{R}$, then its control measure is necessarily given by $F_X$ and
\begin{align*}
\int_\mathbb{R} f(y)\, \Lambda_X (dy) = \int_\mathbb{R} f(y)\, \tilde{\Lambda}_X (dy)
\end{align*}
almost surely for all $f \in L^2(F_X)$.
\end{remark}

\section{Proofs}\label{proofs}
\begin{proof}[Proof of Proposition~\ref{kernels}]
For $\gamma >0$ define $h_\gamma (z) = (-z)^{\gamma}/h(z)$ for each $z\in \mathbb{C}\setminus \{0\}$ with $\text{Re}(z)\leq 0$. By continuity of $h$ and the asymptotics $\vert h_\gamma (z) \vert \sim \vert \eta ([0,\infty))\vert^{-1} \vert z\vert^{\gamma}$, $\vert z \vert \to 0$, and $\vert h_\gamma (z)\vert \sim \vert z \vert^{\gamma - 1}$, $\vert z \vert \to \infty$, it follows that
\begin{align}\label{HardyCond}
\sup_{x<0}\int_\mathbb{R} \vert h_\gamma (x+iy)\vert^2\, dy <\infty
\end{align}
for $\gamma \in (-1/2,1/2)$. In other words, $h_\gamma$ is a certain Hardy function, and thus there exists a function $f_\gamma:\mathbb{R}\to \mathbb{R}$ in $L^2$ which is vanishing on $(-\infty,0)$ and has $\mathcal{L}[f_\gamma](z) = h_\gamma(z)$ when $\text{Re}(z)< 0$, see \cite{contARMAframework,Doetsch,Dym:gaussian}. Note that $f_\gamma$ is indeed real-valued, since $\overline{h_\gamma (x-iy)} = h_\gamma(x+iy)$ for $y\in \mathbb{R}$ and a fixed $x<0$. We can apply \cite[Proposition~2.3]{TaqquFractional} to deduce that there exists a function $g\in L^2$ satisfying (\ref{kernelChar}) and that it can be represented as the (left-sided) Riemann-Liouville fractional integral of $f_{0}$, that is,
\begin{align*}
g(t) = \frac{1}{\Gamma (\beta )} \int_0^t f_{0}(u) (t-u)^{\beta -1}\, du
\end{align*}
for $t>0$. Conversely, \cite[Theorem~2.1]{TaqquFractional} ensures that $D^\beta g$ given by (\ref{dAlphaG}) is a well-defined limit and that $D^\beta g = f_0$. In particular, we have shown (\ref{riemannLiouvilleg}) and if we can argue that $f_{0}\in L^1$, we have shown (\ref{dAlphaLimit}) as well. This follows from the assumption in (\ref{finiteEta}), since then we have that $y \mapsto \mathcal{L}[f_0](x+iy)$ is differentiable for any $x\leq 0$ (except at $0$ when $x=0$) and
\begin{align}\label{diffLaplace}
\begin{aligned}
\mathcal{L}[u\mapsto uf_{0}(u)](x+iy)&=-i\frac{d}{dy} \mathcal{L}[f_0](x+iy) \\
&=\frac{\mathcal{L}[u \, \eta (du)](x+iy)-(1-\beta)(x+iy)^{-\beta}}{h(x+iy)^2}.
\end{aligned}
\end{align}
The function $\mathcal{L}[u\mapsto uf_{0}(u)]$ is analytic on $\{z \in \mathbb{C}\, :\, \text{Re}(z)<0\}$ and from the identity (\ref{diffLaplace}) it is not too difficult to see that it also satisfies the Hardy condition (\ref{HardyCond}). This means $u\mapsto uf_{0}(u)$ belongs to $L^2$, and hence we have that $f_{0}$ belongs to $L^1$. Since $g$ is the Riemann-Liouville integral of $f_0$ of order $\beta$ and $f_0 \in L^1 \cap L^2$, \cite[Proposition~4.3]{basse2018multivariate} implies that $g \in L^\gamma$ for $(1-\beta)^{-1}<\gamma \leq 2$.

It is straightforward to verify (\ref{functionalEquation}) and to obtain the identity
\begin{align*}
\int_s^t \big( D^\beta g\big)\ast \eta (u-\cdot)\, du = \int_\mathbb{R} \big(D^\beta_-\mathds{1}_{(s,t]} \big)(u)\, g \ast \eta (u- \cdot)\, du
\end{align*}
almost everywhere by comparing their Fourier transforms. This establishes the relation
\begin{align*}
g(t-v) - g(s-v) = \int_s^t \big(D^\beta g \big)\ast \eta (u-v)\, du + \mathds{1}_{(s,t]}(v)
\end{align*}

By letting $s \to - \infty$, and using that $D^\beta g$ and $g$ are both vanishing on $(-\infty,0)$, we deduce that
\begin{align*}
g(t) = \mathds{1}_{[0,\infty)}(t)\biggr( 1+ \int_0^t(D^\beta g)\ast \eta (u)\, du  \biggr),
\end{align*}
for almost all $t\in \mathbb{R}$ which shows (\ref{functionalRelation}) and, thus, finishes the proof.
\end{proof}

\begin{proof}[Proof of Theorem~\ref{existenceTheorem}]
Since $g\in L^2$, according to Proposition~\ref{kernels}, and $\mathbb{E}[L_1^2]<\infty$ and $\mathbb{E}[L_1] = 0$,
\begin{align*}
X_t = \int_{-\infty}^t g(t-u)\, dL_u,\quad t \in \mathbb{R},
\end{align*}
is a well-defined process (e.g., in the sense of \cite{Rosinski_spec}) which is stationary with mean zero and finite second moments. By integrating both sides of (\ref{functionalEquation}) with respect to $(L_t)_{t\in \mathbb{R}}$ we obtain
\begin{align*}
X_t -X_s = \int_\mathbb{R}\biggr(\int_\mathbb{R} \big(D_-^\beta \mathds{1}_{(s,t]} \big)(u)\, g\ast \eta (u-r)\, du \biggr)\, dL_r + L_t -L_s.
\end{align*}
By a stochastic Fubini result (such as \cite[Theorem~3.1]{QOU}) we can change the order of integration (twice) and obtain
\begin{align*}
\int_\mathbb{R}\biggr(\int_\mathbb{R} \big(D_-^\beta \mathds{1}_{(s,t]} \big)(u)\, g\ast \eta (u-r)\, du \biggr)\, dL_r = \int_\mathbb{R}\big(D_-^\beta \mathds{1}_{(s,t]} \big)(u)\, X\ast \eta (u)\, du.
\end{align*}
This shows that $(X_t)_{t\in\mathbb{R}}$ is a solution to (\ref{fractionalSDDE}). To show uniqueness, note that the spectral process $\Lambda_X$ of any purely non-deterministic solution $(X_t)_{t\in \mathbb{R}}$ satisfies
\begin{align}\label{simpleRelation}
\int_\mathbb{R}\mathcal{F}[\mathds{1}_{(s,t]}](y)(iy)^\beta h(-iy)\, \Lambda_X (dy) = L_t - L_s
\end{align}
almost surely for any choice of $s<t$ by Theorem~\ref{spectralTheorem} and Proposition~\ref{FubiniRelation}. Using the fact that $(X_t)_{t\in \mathbb{R}}$ is purely non-deterministic, $F_X$ is absolutely continuous with respect to the Lebesgue measure, and hence we can extend (\ref{simpleRelation}) from $\mathds{1}_{(s,t]}$ to any function $f\in L^2$ using an approximation of $f$ with simple functions of the form $s = \sum_{j=1}^n \alpha_j \mathds{1}_{(t_{j-1},t_j]}$ for $\alpha_j \in \mathbb{C}$ and $t_0<t_1<\cdots <t_n$. Specifically, we establish that
\begin{align}\label{generalRelation}
\int_\mathbb{R}\mathcal{F}[f](y)(iy)^\beta h(-iy)\, \Lambda_X (dy) = \int_\mathbb{R} f(u)\, dL_u
\end{align}
almost surely for any $f \in L^2$. In particular we may take $f = g(t-\cdot)$, $g$ being the solution kernel characterized in (\ref{kernelChar}), so that $\mathcal{F}[g(t-\cdot)](y) = e^{ity}(iy)^{-\beta}/h(-iy)$ and (\ref{generalRelation}) thus implies that
\begin{align*}
X_t = \int_{-\infty}^t g(t-u)\, dL_u,
\end{align*}
which ends the proof.

\end{proof}

\begin{proof}[Proof of Proposition~\ref{fractionalX}]
We start by arguing that the limit in (\ref{dAlphaX}) exists and is equal to $\int_{-\infty}^t D^\beta g (t-u)\, dL_u$. For a given $\delta>0$ it follows by a stochastic Fubini result that
\begin{align}\label{FubiniExistence}
\frac{\beta}{\Gamma (1-\beta)}\int_\delta^\infty \frac{X_t-X_{t-u}}{u^{1+\beta}}\, du  = \int_\mathbb{R}D^\beta_\delta g(t-r)\, dL_r,
\end{align}
where
\begin{align*}
D^\beta_\delta g (t) = \frac{\beta}{\Gamma (1- \beta)} \int_\delta^\infty \frac{g(t) - g(t-u)}{u^{1+ \beta}}\, du
\end{align*}
for $t>0$ and $D^\beta_\delta g(t) = 0$ for $t\leq0$. Suppose for the moment that $(L_t)_{t\in \mathbb{R}}$ is a Brownian motion, so that $(X_t)_{t\in \mathbb{R}}$ is $\gamma$-H\"{o}lder continuous for all $\gamma \in (0,1/2)$ by (\ref{fractionalSDDE}). Then, almost surely, $u \mapsto (X_t-X_{t-u})/u^{1+\beta}$ is in $L^1$ and the relation (\ref{FubiniExistence}) thus shows that
\begin{align*}
\int_\mathbb{R}\big[ D_\delta^\beta g(t-r)-D_{\delta'}^\beta g (t-r) \big]\, dL_r \overset{\mathbb{P}}{\to} 0\quad \text{as}\quad \delta,\delta' \to 0,
\end{align*}
which in turn implies that $(D^\beta_\delta g)_{\delta >0}$ has a limit in $L^2$. We also know that this limit must be $D^\beta g$, since $D^\beta_\delta g \to D^\beta g$ pointwise as $\delta \downarrow 0$ by (\ref{dAlphaG}). Having established this convergence, which does not rely on $(L_t)_{t\in \mathbb{R}}$ being a Brownian motion, it follows immediately from (\ref{FubiniExistence}) and the isometry property of the integral map $\int_\mathbb{R} \cdot dL$ that the limit in (\ref{dAlphaX}) exists and that $D^\beta X_t = \int_{-\infty}^t D^\beta g (t-u)\, dL_u$. To show (\ref{fractionalChange}) we start by recalling the definition of $D^\beta_- \mathds{1}_{(s,t]}$ in (\ref{fracIndicator}) and that $\mathcal{F}[D^\beta_- \mathds{1}_{(s,t]}](y) = (iy)^\beta \mathcal{F}[\mathds{1}_{(s,t]}](y)$. This identity can be shown by using that the improper integral $\int_0^\infty e^{\pm iv} v^{\gamma -1}\, dv$ is equal to $\Gamma (\gamma) e^{\pm i\pi \gamma/2}$ for any $\gamma \in (0,1)$. Now observe that
\begin{align*}
\mathcal{F}\biggr[\int_\mathbb{R} \big(D^\beta_- \mathds{1}_{(s,t]}\big)(u)\, g\ast \eta (u-\cdot)\, du \biggr](y) &= (iy)^\beta \mathcal{F}[\mathds{1}_{(s,t]}](y) \mathcal{F}[g](-y) \mathcal{F}[\eta](-y)\\
&= \mathcal{F}[\mathds{1}_{(s,t]}](y) \mathcal{F}\big[\big(D^\beta g\big)\ast \eta \big] (-y)\\
&= \mathcal{F}\biggr[\int_s^t (D^\beta g\big)\ast \eta (u-\cdot) \, du \biggr](y),
\end{align*}
and hence $\int_\mathbb{R} \big(D^\beta_- \mathds{1}_{(s,t]}\big)(u)\,  g\ast \eta (u-\cdot)\, du = \int_s^t (D^\beta g\big)\ast \eta (u-\cdot) \, du$ almost everywhere. Consequently, using that $D^\beta X_t = \int_{-\infty}^t D^\beta g (t-u)\, dL_u$ and applying a stochastic Fubini result twice,
\begin{align*}
\int_s^t \big(D^\beta X \big)\ast \eta (u)\, du &= \int_\mathbb{R}\int_s^t \big(D^\beta g \big)\ast \eta (u-r)\, du \, dL_r\\
&= \int_\mathbb{R}\int_\mathbb{R} \big(D^\beta_- \mathds{1}_{(s,t]} \big)(u)\, g \ast \eta (u-r)\, du \, dL_r\\
&= \frac{1}{\Gamma (1-\beta)}\int_\mathbb{R} \big[(t-u)_+^{-\beta} - (s-u)_+^{-\beta} \big]\, X\ast \eta (u)\, du.
\end{align*}
The semimartingale property of $(X_t)_{t\in \mathbb{R}}$ is now an immediate consequence of (\ref{fractionalSDDE}).
\end{proof}

\begin{proof}[Proof of Proposition~\ref{memory}]
Using (\ref{spectralDensity}) and that $h(0)= - \eta ([0,\infty))$, it follows that $f_X(y) \sim \vert y \vert^{-2 \beta}/\eta ([0,\infty))^2$ as $y \to 0$. To show the asymptotic behavior of $\gamma_X$ at $\infty$ we start by recalling that, for $u,v \in \mathbb{R}$,
\begin{align*}
\int_{u \vee v}^\infty (s-u)^{\beta-1}(s-v)^{\beta-1}\, ds = \frac{\Gamma (\beta)\Gamma (1-2\beta)}{\Gamma (1-\beta)} \vert u- v\vert^{2\beta -1}
\end{align*}
by \cite[p.~404]{gripenberg1996prediction}. Having this relation in mind we use Proposition~\ref{kernels}(\ref{riemannLiouvilleg}) and (\ref{autocovariance}) to do the computations
\begin{align}
\gamma_X (t) =&\, \frac{1}{\Gamma (\beta)^2}\int_\mathbb{R}\int_\mathbb{R}\int_\mathbb{R} D^\beta g(u) D^\beta g(v)(s+t-u)_+^{\beta-1}(s-v)_+^{\beta -1}\, dv \, du\, ds \notag\\
=&\, \frac{1}{\Gamma (\beta)^2}\int_\mathbb{R}\int_\mathbb{R} D^\beta g(u) D^\beta g(v) \notag\\
&\cdot\, \int_{(u-t)\vee v}^\infty (s-(u-t))^{\beta-1}(s-v)^{\beta -1}\, ds\,  dv\, du \notag \\
=&\, \frac{\Gamma (1-2\beta)}{\Gamma (\beta)\Gamma(1-\beta)} \int_\mathbb{R}\int_\mathbb{R}D^\beta g (u)D^\beta g (v) \vert u-v-t \vert^{2\beta -1}\, dv \, du \notag \\
=&\, \frac{\Gamma (1-2\beta)}{\Gamma (\beta)\Gamma(1-\beta)} \int_\mathbb{R} \gamma (u)\vert u-t\vert^{2\beta -1}\, du, \label{autoCovProof}
\end{align}
where $\gamma (u)= \int_\mathbb{R}D^\beta g(u+v) D^\beta g (v)\, dv$. Note that $\gamma \in L^1$ since $D^\beta g \in L^1$ by Proposition~\ref{kernels} and, using Plancherel's theorem,
\begin{align*}
\gamma (u) = \int_\mathbb{R}e^{iuy} \big\vert \mathcal{F}\big[D^\beta g\big](y) \big\vert^2 \, dy =  \mathcal{F}\big[\vert h(i\cdot)\vert^{-2} \big] (u).
\end{align*}
In particular $\int_\mathbb{R} \gamma (u)\, du = \vert h(0)\vert^{-2} = \eta ([0,\infty))^{-2}$, and hence it follows from (\ref{autoCovProof}) that we have shown the result if we can argue that
\begin{align}\label{shortGammaConv}
\frac{\int_\mathbb{R} \gamma (u) \vert u - t \vert^{2\beta -1}\, du }{t^{2\beta - 1}} = \int_\mathbb{R}\frac{\gamma (u)}{\vert \tfrac{u}{t}-1 \vert^{1-2\beta}}\, du \to \int_\mathbb{R}\gamma (u)\, du,  \quad t \to \infty.
\end{align}
It is clear by Lebesgue's theorem on dominated convergence that 
\begin{align*}
\int_{-\infty}^0 \frac{\gamma (u)}{\vert\tfrac{u}{t}-1\vert^{1-2\beta}}\, du \to \int_{-\infty}^0 \gamma (u)\, du,\quad t \to \infty.
\end{align*}
Moreover, since $\vert h(i\cdot)\vert^{-2}$ is continuous at $0$ and differentiable on $(-\infty,0)$ and $(0,\infty)$ with integrable derivatives, it is absolutely continuous on $\mathbb{R}$ with a density $\phi$ in $L^1$. As a consequence, $\gamma (u) = \mathcal{F}[\phi](y)/(-iy)$ and, thus,
\begin{align}\label{intermed}
\int_{t/2}^\infty \frac{\gamma (u)}{\vert \tfrac{u}{t}-1\vert^{1-2\beta}}\, du = \int_{1/2}^\infty \frac{t\gamma (tu)}{\vert u-1\vert^{1-2\beta}}\, du = i\int_{1/2}^\infty \frac{\mathcal{F}[\phi](tu)}{u \vert u-1\vert^{1-2\beta}}\, du.
\end{align}
By the Riemann-Lebesgue lemma and Lebesgue's theorem on dominated convergence it follows that the right-hand side of expression in (\ref{intermed}) tends to zero as $t$ tends to infinity. Finally, integration by parts and the symmetry of $\gamma$ yields
\begin{align*}
\int_0^{t/2} \gamma (u) \biggr(1 - \frac{1}{\vert \tfrac{u}{t} - 1 \vert^{1- 2\beta}} \biggr)\, du =&\, \int_0^{1/2} t \gamma (tu)\biggr(1 - \frac{1}{(1-u)^{1-2\beta}} \biggr)\, du\\
=&\, \big(2^{1-2\beta} - 1\big)\int_{-\infty}^{-t/2}\gamma (u)\, du  \\
&- \int_0^{1/2} \frac{1- 2\beta}{(1-u)^{2-2\beta}} \int_{-\infty}^{-tu} \gamma (v)\, dv\, du,
\end{align*}
where both terms on the right-hand side converge to zero as $t$ tends to infinity. Thus, we have shown (\ref{shortGammaConv}), and this completes the proof.
\end{proof}

\begin{proof}[Proof of Proposition~\ref{roughness}] Observe that it is sufficient to argue $\mathbb{E}[(X_t-X_0)^2] \sim t$ as $t \downarrow 0$. By using the spectral representation $X_t = \int_\mathbb{R} e^{ity}\, \Lambda_X (dy)$ and the isometry property of the integral map $\int_\mathbb{R} \cdot d\Lambda_X :L^2(F_X) \to L^2(\mathbb{P})$, see \cite[p.~389]{grimmett2001probability}, we have that
\begin{align}\label{keyRelationRough}
\frac{\mathbb{E}\big[(X_t -X_0)^2 \big]}{t} &= t^{-2}\int_\mathbb{R}\big\vert 1- e^{iy}\big\vert^2 f_X (y/t)\, dy \notag \\
&= \int_\mathbb{R}\frac{\big\vert 1- e^{iy} \big\vert^2}{\vert y \vert^{2\beta}\big\vert (-iy)^{1-\beta} - t^{1-\beta} \mathcal{F}[\eta](y/t) \big\vert^2}\, dy.
\end{align}
Consider now a $y \in \mathbb{R}$ satisfying $\vert y \vert \geq C_1 t$ with $C_1:= (2 \vert \eta \vert ([0,\infty)))^{1/(1-\beta)}$. In this case $\vert y \vert^{1-\beta}/2 - \vert t^{1-\beta}\mathcal{F}[\eta](y/t)\vert\geq 0$, and we thus get by the reversed triangle inequality that
\begin{align*}
\frac{\big\vert 1- e^{iy}\big\vert^2}{\vert y \vert^{2\beta}\big\vert (-iy)^{1-\beta}-t^{1-\beta}\mathcal{F}[\eta](y/t)\big\vert^2} \leq  2\frac{\big\vert 1- e^{iy}\big\vert^2}{y^2}.
\end{align*}
If $\vert y\vert < C_1 t$, we note that the assumption on the function in (\ref{hFunction}) implies that 
\begin{align*}
C_2 := \inf_{\vert x \vert \leq C_1} \big\vert (-ix)^{1-\beta}- \mathcal{F}[\eta](x) \big\vert >0,
\end{align*}
which shows that
\begin{align*}
\big\vert (-iy)^{1-\beta} - t^{1-\beta} \mathcal{F}[\eta](y/t)\big\vert \geq t^{1-\beta} C_2 \geq \frac{C_2}{C_1^{1-\beta}} \vert y \vert^{1-\beta}.
\end{align*}
This establishes that
\begin{align*}
\frac{\big\vert 1-e^{iy} \big\vert^2}{\vert y \vert^{2\beta} \big\vert (-iy)^{1-\beta} - t^{1-\beta} \mathcal{F}[\eta](y/t) \big\vert^2} \leq \frac{C_1^{2(1-\beta)}}{C_2^2} \frac{\big\vert 1- e^{iy} \big\vert^2}{y^2}.
\end{align*}
Consequently, it follows from (\ref{keyRelationRough}) and Lebesgue's theorem on dominated convergence that
\begin{align*}
\frac{\mathbb{E}\big[(X_t - X_0)^2 \big]}{t} \to \int_\mathbb{R} \frac{\big\vert 1 - e^{iy} \big\vert^2}{y^2}\, dy = \int_\mathbb{R} \vert \mathcal{F}[\mathds{1}_{(0,1]}](y)\vert^2\, dy = 1
\end{align*}
as $h \downarrow 0$, which was to be shown.
\end{proof}

\begin{proof}[Proof of Proposition~\ref{prediction}]
We start by arguing that the first term on the right-hand side of the formula is well-defined. In order to do so it suffices to argue that
\begin{align}\label{EstimateFinite}
\begin{aligned}
\MoveEqLeft\EE  \biggr[\int_{0}^{t-s} \int_{-\infty}^s \vert X_w \vert    \int_{[0,\infty)} \big\vert \big(D_-^\beta \mathds{1}_{(s,t-u]}\big)(v+w)\big\vert  \, \vert \eta \vert (dv) \, dw \, \vert g \vert (du)\biggr] \\
& \leq \mathbb{E}[\vert X_0\vert] \int_{0}^{t-s}\int_{[0,\infty)}  \int_{-\infty}^s  \big\vert \big(D_-^\beta \mathds{1}_{(s,t-u]}\big)(v+w)\big\vert  \, dw  \, \vert \eta \vert (dv) \, \vert g \vert (du)
\end{aligned}
\end{align}
is finite. This is implied by the facts that
\begin{align*}
 \MoveEqLeft\Gamma (1-\beta)\int_{-\infty}^s  \big\vert \big(D_-^\beta\mathds{1}_{(s,t-u]}\big)(v+w)\vert  \, dw \\
  \leq & \int_{u+s-t}^0  (t-s-u+w)^{-\beta} \, dw + \int_0^1  w^{-\beta} - (t-s-u+w)^{-\beta}  \, dw\\
  &+(1+\beta) \int_1^\infty  w^{-1-\beta} (t-s-u)  \, dw\\
 =& \frac{1}{1-\beta} \big(2(t-s-u)^{1-\beta} + 1 - (t-s-u+1)^{1-\beta}\big) + \frac{(1+\beta)}{\beta} (t-s-u) \\
 \leq & \frac{2}{1-\beta} (t-s)^{1-\beta}  + \frac{(1+\beta)}{\beta} (t-s)
\end{align*}
for $u \in [0,t-s]$ and $g(du)$ is a finite measure (since $D^\beta g \in L^1$ by Proposition~\ref{kernels}). Now fix an arbitrary $z \in \mathbb{C}$ with $\text{Re}(z)<0$. It follows from (\ref{fractionalSDDE}) that 
\begin{align}\label{LaplaceOfEq}
\begin{aligned}
\mathcal{L}[X\mathds{1}_{(s,\infty)}](z) =& X_s\mathcal{L}[\mathds{1}_{(s,\infty}](z)  + \mathcal{L}[\mathds{1}_{(s,\infty)} (L_\cdot - L_s)](z)\\
&+ \mathcal{L}\biggr[\mathds{1}_{(s,\infty)}\int_\mathbb{R} X_u \int_{[0,\infty)}\big(D^\beta_- \mathds{1}_{(s,\cdot]} \big)(u+v)\,  \eta (dv)\, du \biggr] (z).
\end{aligned}
\end{align}
By noting that $(D^\beta_-\mathds{1}_{(s,t]}) (u) =0$ when $t\leq s <u$ we obtain
\begin{align*}
\MoveEqLeft\mathcal{L}\biggr[\mathds{1}_{(s,\infty)}\int_s^\infty X_u \int_{[0,\infty)}\big(D^\beta_- \mathds{1}_{(s,\cdot]} \big)(u+v)\,  \eta (dv)\, du \biggr] (z) \\
&=\frac{1}{\Gamma (1-\beta)}\LL\biggr[ \int_s^\infty X_u  \int_{[0,\infty)} (\cdot - u-v)^{-\beta}_+ \, \eta(dv) \, du\biggr](z) \\
 &= \LL[ \mathds{1}_{(s,\infty)}X](z) \LL[\eta](z) (-z)^{\beta-1}. 
\end{align*}
Combining this observation with \eqref{LaplaceOfEq} we get the relation
\begin{align*}
\MoveEqLeft \big(-z -   (-z)^{\beta}\LL[\eta](z) \big)\LL[\mathds{1}_{(s,\infty)}X](z) \\
=&X_s (-z)\LL[ \mathds{1}_{(s,\infty)}](z)+ (-z) \LL[\mathds{1}_{(s,\infty)}(L-L_s)](z)\\
&+ (-z) \LL\biggr[\mathds{1}_{(s,\infty)}  \int_{-\infty}^s X_u  \int_{[0,\infty)} \big(D^\beta_-\mathds{1}_{(s,\cdot]}\big)(u+v)\, \eta(dv) \, du\biggr](z), 
\end{align*}
which implies
\begin{align*}
\MoveEqLeft\LL[\mathds{1}_{(s,\infty)}X](z) \\
=& \LL[g](z) \LL[X_s \delta_0 (s-\cdot)](z) + \LL[g](z)(-z) \LL[\mathds{1}_{(s,\infty)}(L-L_s)](z)\\
&+ \LL[g](z) (-z)  \LL\biggr[\mathds{1}_{(s,\infty)} \int_{-\infty}^s X_u  \int_{[0,\infty)} \big(D^\beta_-\mathds{1}_{(s,\cdot]}\big)(u+v) \, \eta(dv) \, du\biggr](z) \\
=& \LL[g(\cdot - s)X_s](z) + \LL \biggr[ \int_s^\cdot g(\cdot-u)dL_u\biggr](z) \\
&+ \LL\biggr[\int_{0}^{\cdot-s} \int_{-\infty}^s X_w   \int_{[0,\infty)}\big(D^\beta_-\mathds{1}_{(s,\cdot-u]}\big)(v+w) \, \eta(dv) \, dw \, g(du)\biggr](z).
\end{align*}
This establishes the identity
\begin{align} \label{varOfConst}
\begin{aligned}
X_t =& g(t-s) X_s + \int_s^t g(t-u)\, dL_u\\
&+\int_0^{t-s} \int_{-\infty}^s X_w \int_{[0,\infty)} \big(D^\beta_-\mathds{1}_{(s,t-u]} \big) (v+w)\, \eta (dv)\, dw\, g(du)
\end{aligned}
\end{align}
almost surely for Lebesgue almost all $t>s$. Since both sides of (\ref{varOfConst}) are continuous in $L^1 (\mathbb{P})$, the identity holds for each fixed pair $s<t$ almost surely as well. By applying the conditional mean $\mathbb{E}[\cdot \mid X_u,\, u \leq s]$ on both sides of (\ref{varOfConst}) we obtain the result.
\end{proof}

\begin{proof}[Proof of Corollary~\ref{exponentialExistence}]
In this setup it follows that the function $h$ in (\ref{hFunction}) is given by
\begin{align*}
h(z) = (-z)^{1-\beta} +\kappa + \frac{R(-z)}{Q(-z)},
\end{align*}
where $Q(z) \neq 0$ whenever $\text{Re}(z) \geq 0$ by the assumption on $A$. This shows that $h$ is non-zero (on  $\{z\in \mathbb{C}\, :\, \text{Re}(z)\leq 0\}$) if and only if
\begin{align}\label{expCondition}
Q(z)\big[ z^{1-\beta} + \kappa \big]   + R(z) \neq 0\quad \text{for all } z \in \mathbb{C} \text{ with } \text{Re}(z) \geq 0.
\end{align}
Condition (\ref{expCondition}) may equivalently be formulated as $Q(z)[z + \kappa z^\beta] + R(z) z^\beta \neq 0$ for all $z\in \mathbb{C}\setminus \{0\}$ with $\text{Re}(z) \geq 0$ and $h(0) = \kappa + b^TA^{-1}e_1\neq 0$, which by Theorem~\ref{existenceTheorem} shows that a unique solution to (\ref{exponentialSDDE}) exists. It also provides the form of the solution, namely (\ref{solution}) with
\begin{align*}
\mathcal{F}[g](y) &= \frac{(-iy)^{-\beta}}{(-iy)^{1-\beta} +\kappa + \frac{R(-iy)}{Q(-iy)}} \\
&= \frac{Q(-iy)}{Q(-iy)\big[-iy + \kappa (-iy)^\beta \big] + R(-iy) (-iy)^\beta}
\end{align*}
for $y \in \mathbb{R}$. This finishes the proof.
\end{proof}

\begin{proof}[Proof of Proposition~\ref{fracDiffChange}]
We will first show that $D^\beta f \in L^1$. By using that $\int_0^\infty e^{Au}\, du = - A^{-1}$ we can rewrite $D^\beta f$ as 
\begin{align*}
D^\beta f (t) = \frac{1}{\Gamma (1-\beta)}b^T A\biggr(\int_0^t e^{Au}\big[(t-u)^{-\beta}-t^{-\beta} \big]\, du - \int_t^\infty e^{Au}t^{-\beta}\, du\biggr) e_1
\end{align*}
for $t>0$, from which we see that it suffices to argue that (each entry of)
\begin{align*}
t \mapsto \int_0^t e^{Au}\big[(t-u)^{-\beta}-t^{-\beta} \big]\, du
\end{align*}
belongs to $L^1$. Since $u \mapsto e^{Au}$ is continuous and with all entries decaying exponentially fast as $u \to \infty$, this follows from the fact that, for a given $\gamma >0$,
\begin{align*}
\MoveEqLeft\int_0^\infty \int_0^t e^{-\gamma u} \big\vert (t-u)^{-\beta}-t^{-\beta} \big\vert\, du\, dt\\
&\leq \int_0^\infty e^{-\gamma u}\biggr(\int_u^{u+1}\big((t-u)^{-\beta}+t^{-\beta} \big)\, dt + \beta u \int_1^\infty t^{-\beta -1}\, dt\biggr)\, du < \infty.
\end{align*}
Here we have used the mean value theorem to establish the inequality
\begin{align*}
\big\vert (t-u)^{-\beta}-t^{-\beta} \big\vert \leq \beta u( t-u)^{-\beta-1}
\end{align*}
for $0<u<t$. To show that $D^\beta f\in L^2$, note that it is the left-sided Riemann-Liouville fractional derivative of $f$, that is,
\begin{align*}
D^\beta f (t) &= \frac{1}{\Gamma (1-\beta)} \frac{d}{dt}\int_0^t f(t-u)u^{-\beta}\, du
\end{align*}
for $t>0$.
Consequently, it follows by \cite[Theorem~7.1]{samko1993fractional} that the Fourier transform $\mathcal{F}[D^\beta f]$ of $f$ is given by
\begin{align*}
\mathcal{F}\big[D^\beta f\big](y) = (-iy)^\beta \mathcal{F}[f](y)  = -(-iy)^\beta b^T(A+iy)^{-1}e_1,\quad y \in \mathbb{R},
\end{align*}
in particular it belongs to $L^2$ (e.g., by Cramer's rule), and thus $D^\beta f \in L^2$. By comparing Fourier transforms we establish that $(D^\beta g)\ast f = g \ast (D^\beta f)$, and hence it holds that
\begin{align*}
\int_0^\infty D^\beta X_{t-u} f(u)\, du = \int_\mathbb{R}\big(D^\beta g \big)\ast f (t-r) \, dL_r = \int_0^\infty X_{t-u}D^\beta f(u)\, du
\end{align*}
using Proposition~\ref{fractionalX} and a stochastic Fubini result. This finishes the proof.
\end{proof}

\begin{proof}[Proof of Proposition~\ref{stochFubini}]
First, note that (\ref{FubiniRelation}) is trivially true when $f$ is of the form
\begin{align}\label{simplef}
f(x,y) = \sum_{j=1}^n \alpha_j \mathds{1}_{A_j}(x) \mathds{1}_{B_j}(y)
\end{align}
for $\alpha_1,\dots, \alpha_n \in \mathbb{C}$ and Borel sets $A_1,B_1,\dots, A_n,B_n\subseteq \mathbb{R}$. Now consider a general $f\in L^2(\mu \times G)$ and choose a sequence of functions $(f_n)_{n\in \mathbb{N}}$ of the form (\ref{simplef}) such that $f_n \to f$ in $L^2(\mu \times G)$ as $n\to \infty$. Set
\begin{align*}
&X_n = \int_\mathbb{R}\biggr(\int_\mathbb{R}f_n(x,y)\, \mu (dx) \biggr)\, S(dy),\ X = \int_\mathbb{R} \biggr(\int_\mathbb{R}f(x,y)\, \mu (dx) \biggr)\, S(dy) \\
\text{and}\quad & Y= \int_\mathbb{R}\biggr(\int_\mathbb{R}f(x,y)\, S(dy) \biggr)\, \mu (dx).
\end{align*} 
Observe that $X$ and $Y$ are indeed well-defined, since $x\mapsto f(x,y)$ is in $L^1(\mu)$ for $G$-almost all $y$, $y\mapsto f(x,y)$ is in $L^2(G)$ for $\mu$-almost all $x$,
\begin{align*}
&\int_\mathbb{R}\biggr\vert \int_\mathbb{R}f(x,y)\, \mu (dx) \biggr\vert^2\, G(dy) \leq \mu (\mathbb{R})\int_{\mathbb{R}^2} \vert f(x,y)\vert^2\, (\mu \times G)(dx,dy)< \infty\\
\text{and}\quad & \mathbb{E}\biggr[\int_\mathbb{R}\biggr\vert \int_\mathbb{R} f(x,y)\, S(dy) \biggr\vert^2 \, \mu (dx)\biggr] = \int_{\mathbb{R}^2} \vert f(x,y)\vert^2\, (\mu \times G)(dx,dy) < \infty.
\end{align*}
Next, we find that
\begin{align*}
\mathbb{E}[\vert X-X_n \vert^2] &= \int_\mathbb{R}\biggr\vert \int_\mathbb{R} (f(x,y)-f_n(x,y) \, \mu (dx)\biggr\vert^2\, G(dy) \\
&\leq \mu (\mathbb{R}) \int_{\mathbb{R}^2} \vert f(x,y)-f_n(x,y)\vert^2\, (\mu \times G)(dx,dy)
\end{align*}
which tends to zero by the choice of $(f_n)_{n\in \mathbb{N}}$. Similarly, using that $X_n = \int_\mathbb{R}\big(\int_\mathbb{R}f_n(x,y)\, S(dy) \big)\, \mu (dx)$, one shows that $X_n \to Y$ in $L^2(\mathbb{P})$, and hence we conclude that $X=Y$ almost surely.
\end{proof}
\begin{proof}[Proof of Theorem~\ref{spectralTheorem}]
For any given $t\in \mathbb{R}$ set $f_t(y) = e^{ity}$, $y \in \mathbb{R}$, and let $H_F$ and $H_X$ be the set of all (complex) linear combinations of $\{f_t\, :\, t \in \mathbb{R}\}$ and $\{X_t\, :\, t \in \mathbb{R}\}$, respectively. By equipping $H_F$ and $H_X$ with the usual inner products on $L^2(F_X)$ and $L^2(\mathbb{P})$, their closures $\overline{H_F}$ and $\overline{H_X}$ are Hilbert spaces. Due to the fact that
\begin{align*}
\langle X_s,X_t \rangle_{L^2(\mathbb{P})}=\mathbb{E}[X_sX_t] = \int_\mathbb{R}e^{i(t-s)x}\, F_X(dy) = \langle f_s,f_t \rangle_{L^2(F_X)},\quad s,t \in \mathbb{R},
\end{align*}
we can define a linear isometric isomorphism $\mu:\overline{H_F} \to \overline{H_X}$ as the one satisfying
\begin{align*}
\mu \biggr( \sum_{j=1}^n\alpha_j f_{t_j} \biggr) = \sum_{j=1}^n \alpha_j X_{t_j}
\end{align*}
for any given $n\in \mathbb{N}$, $\alpha_1,\dots, \alpha_n \in \mathbb{C}$ and $t_1<\cdots < t_n$. Since $\mathds{1}_{(-\infty,y]} \in \overline{H_F}$ for each $y\in \mathbb{R}$, cf. \cite[p. 150]{Yaglom:cor}, we can associate a (complex-valued) process $\Lambda_X = \{\Lambda_X(y)\, :\, y\in \mathbb{R}\}$ to $(X_t)_{t\in \mathbb{R}}$ through the relation
\begin{align*}
\Lambda_X (y) = \mu (\mathds{1}_{(-\infty,y]}).
\end{align*}
It is straight-forward to check from the isometry property that $\Lambda_X$ is right-continuous in $L^2(\mathbb{P})$, has orthogonal increments and satisfies
\begin{align*}
\mathbb{E}\big[\vert \Lambda_X(y_2) - \Lambda_X(y_1)\vert^2\big] = F_X((y_1,y_2])
\end{align*}
for $y_1<y_2$. Consequently, integration with respect to $\Lambda_X$ of any function $f \in L^2(F_X)$ can be defined in the sense of \cite[pp 388-390]{grimmett2001probability}. For any $n\in \mathbb{N}$, $\alpha_1,\dots, \alpha_n \in \mathbb{C}$ and $t_0<t_1<\cdots < t_n$, we have 
\begin{align*}
\int_\mathbb{R}\biggr(\sum_{j=1}^n\alpha_j\mathds{1}_{(t_{j-1},t_j]}(y) \biggr)\, \Lambda_X(dy) = \sum_{j=1}^n \alpha_j \mu (\mathds{1}_{(t_{j-1},t_j]}) = \mu \biggr(\sum_{j=1}^n \alpha_j \mathds{1}_{(t_{j-1},t_j]} \biggr).
\end{align*}
Since $f \mapsto \int_\mathbb{R}f(y)\, \Lambda_X(dy)$ is a continuous map (from $L^2(F_X)$ into $L^2(\mathbb{P})$), it follows by approximation with simple functions and from the relation above that
\begin{align*}
\int_\mathbb{R} f(y)\, \Lambda_X (dy) = \mu (f)
\end{align*}
almost surely for any $f\in \overline{H_F}$. In particular, it shows that
\begin{align*}
X_t = \mu (f_t) = \int_\mathbb{R}e^{ity}\, \Lambda_X (dy),\quad t \in\mathbb{R},
\end{align*}
which is the spectral representation of $(X_t)_{t\in \mathbb{R}}$. 
\end{proof}

\subsection*{Acknowledgments}
The authors thank Andreas Basse-O'Connor and Jan Pedersen for helpful comments. The research of Richard Davis was supported in part by ARO MURI grant W911NF-12-1-0385. The research of Mikkel Slot Nielsen and Victor Rohde was supported by Danish Council for Independent Research grant DFF-4002-00003.

\bibliographystyle{imsart-nameyear}

\end{document}